\documentclass[11pt,letterpaper,reqno,oneside]{amsart}
\usepackage[letterpaper, includehead, includefoot,top=30mm, bottom=30mm, left=30mm, right=30mm,marginparwidth=25mm, marginparsep=4mm]{geometry}
\usepackage{comment}
\usepackage{times}
\usepackage{amsfonts} %
\usepackage{amsmath} %
\usepackage{amssymb} %
\usepackage{enumerate} %
\usepackage{amsthm} %
\usepackage[all]{xy}
\usepackage{bbm} %
\usepackage{graphicx} %
\usepackage{nicefrac} %
\usepackage{color} %
\usepackage{hyperref}
\usepackage{ulem}
\usepackage{thmtools}
\usepackage{enumerate}
\usepackage{booktabs}
\usepackage{caption}

\usepackage[all]{xy}

\usepackage{color}
\usepackage{hyperref}
\hypersetup{
    colorlinks=true,
    linkcolor=blue,
    citecolor=blue,
    filecolor=blue,
    urlcolor=blue
}
\usepackage[nameinlink,capitalize,noabbrev]{cleveref}

\usepackage[nameinlink,capitalize,noabbrev]{cleveref}

\usepackage{calrsfs}

\usepackage{tikz}
\usetikzlibrary{calc,intersections}
\usepackage{pgfplots}
\usepgfplotslibrary{patchplots}
\usetikzlibrary{patterns, positioning, arrows}
\pgfplotsset{compat=1.15}

\usepackage{todonotes}

\newtheorem{theorem}{Theorem}[section] %
\newtheorem{corollary}[theorem]{Corollary} %
\newtheorem{lemma}[theorem]{Lemma} %
\newtheorem{proposition}[theorem]{Proposition} %
{\theoremstyle{remark} %
  \newtheorem{remark}[theorem]{Remark}} %
{\theoremstyle{definition} %
  \newtheorem{definition}[theorem]{Definition} %
  \newtheorem{example}[theorem]{Example} %
}

\newcommand{\PP}[0]{\ensuremath{\mathbb{P}}}
\newcommand{\CC}[0]{\ensuremath{\mathbb{C}}}
\newcommand{\ZZ}[0]{\ensuremath{\mathbb{Z}}}

\newcommand{\Stab}[0]{\ensuremath{\operatorname{Stab}}}
\newcommand{\codim}[0]{\ensuremath{\operatorname{codim}}}

\newcommand{\lcm}[0]{\ensuremath{\operatorname{lcm}}}

\newcommand{\Fix}[0]{\ensuremath{\operatorname{Fix}}}
\newcommand{\Aut}[0]{\ensuremath{\operatorname{Aut}}}

\newcommand{\diag}[0]{\ensuremath{\operatorname{diag}}}

\newcommand{\PGL}[0]{\ensuremath{\operatorname{PGL}}}

\newcommand{\Id}[0]{\ensuremath{\operatorname{Id}}}

\newcommand{\age}[0]{\ensuremath{\operatorname{age}}}

\newcommand{\AF}[0]{\ensuremath{\mathbb{A}}}
\newcommand{\NN}[0]{\ensuremath{\mathbb{N}}}

\newcommand{\OO}[0]{\ensuremath{\mathcal{O}}}

\newcommand{\tX}[0]{{\ensuremath{\widetilde{X}}}}

\newcommand{\IP}{\mathbb{P}}
\newcommand{\IZ}{\mathbb{Z}}

\DeclareMathOperator{\Spec}{Spec}

\begin{document}

\title{Quotients of Calabi--Yau Generalized Fermat manifolds and Log~Enriques varieties}

\author[S. Billi]{Simone Billi}
\address{Fakult\"at f\"ur Mathematik und Informatik, Universit\"at des Saarlandes, Campus E2.4,
66123 Saarbr\"ucken, Germany} %
\email{billi@math.uni-sb.de}

\author[A. J. Palomino]{Ana Julisa Palomino}
\address{Instituto de Matem\'aticas, Universidad de Talca,
Casilla 721, Talca, Chile} %
\email{ana.palomino@utalca.cl}

\author[A. Sarti]{Alessandra Sarti}
\address{Laboratoire de Math\'ematiques et applications, Universit\'e de Poitiers, UMR 7348 CNRS, 
Poitiers, France} %
\email{alessandra.sarti@univ-poitiers.fr}





\maketitle


\begin{abstract}
Enriques manifolds and Log Enriques varieties have attracted a lot of interest in the last years. In this paper we study Enriques manifolds and Log Enriques varieties that are respectively \'etale and quasi-\'etale quotients of Calabi--Yau manifolds. We produce several examples by studying quotients of generalized Fermat manifolds of Calabi--Yau type. These are complete intersections of Fermat type manifolds all of the same degree and were described recently by Hidalgo, Hughes and Leyton--Álvarez. We describe all their quotients by a fixed-point-free automorphism. Most of the quotients are Log Enriques varieties or Calabi--Yau varieties with terminal singularities. In one case we get a family of Enriques manifolds which is a subfamily of a family described previously by Boissi\`ere, Camere and the third author, which is one of the two so far known families of Enriques manifolds of Calabi--Yau type. This family is obtained as an \'etale quotient of order two of a Calabi--Yau manifold complete intersection of quadrics, the other family is the Hilbert scheme of $n$ points on an Enriques surface. In the last part of the paper we show that these two families are different by computing their second Betti numbers, which are equal to one and eleven respectively.
\end{abstract}

\section{Introduction}
Recall that an Enriques surface is a compact complex surface with non-trivial but $2$-torsion canonical divisor and zero irregularity. It is a classical result that Enriques surfaces can be all obtained as \'etale quotients of K3 surfaces. More precisely, let $S$ be a K3 surface with a fixed-point-free involution $\iota$ then the quotient $S/\langle \iota \rangle$ is an Enriques surface. It turns out, by using Lefschetz formula, that only involutions can act without fixed points on K3 surfaces and the action is non-symplectic, i.e. $\iota^*\omega_S=-\omega_S$ where $\omega_S$ denotes the unique (up to scalar) never zero global holomorphic $2$-form on the K3 surface. Moreover, by \cite{Nikulin}, a K3 surface with a non-symplectic involution is projective so that all Enriques surfaces are projective. 

The notion of Enriques surface was generalized in higher dimension independently by \cite{OS} and \cite{BNWS}, we will follow the more recent and more general definition of \cite{LogEnriques}: an {\it Enriques manifold} is a complex projective manifold $E$ with torsion but non-trivial canonical divisor and zero irregularity. The smallest positive integer $I$ such that 
$I K_E\simeq 0$ is called the {\it index} of the Enriques manifold and observe that by definition $I>1$. This definition naturally leads us to consider quotients of irreducible holomorphic symplectic manifolds but also Calabi--Yau manifolds, the latter motivated by the higher-dimensional generalization of the classical construction of Enriques surfaces. Namely, one starts with complete intersection K3 surfaces in $\mathbb{P}^5$ equipped with a fixed-point-free involution \cite[6.1.5]{LogEnriques},
then the quotient is an Enriques surface and we know moreover that each Enriques surface can be obtained in this way.

By \cite{OS}, \cite{BNWS}, and \cite{LogEnriques}, there exist Enriques manifolds of index $2,3,4$ and no more examples are known at the time being. The most common techniques to produce examples are by considering quotients of irreducible holomorphic symplectic manifolds or Calabi--Yau manifolds by fixed-point-free automorphisms. Observe that a computation with the Lefschetz formula (see \cref{Enriques_manifold_index2} and also \cite{BNWS}) shows that Calabi--Yau manifolds can produce Enriques manifolds only of index $2$, and that the dimension of the Calabi--Yau manifold can only be even.  

A general motivation for the study of Enriques manifolds and their singular counterparts, the Log~Enriques varieties, comes from the Beauville--Bogomolov decomposition theorem for compact complex K\"ahler manifolds with vanishing first Chern class. It asserts that, up to a finite \'etale cover, such a manifold decomposes as a product of irreducible holomorphic symplectic (IHS) manifolds, simply connected Calabi--Yau (CY) manifolds, and complex tori (see \cite{BK}). This was then generalized to the singular setting by H\"oring--Peternell see \cite{Hoering_Peternell}.

The aim of this paper is to study a special family of manifolds called {\it Generalized Fermat manifolds (GFM)} defined in \cref{GFM} and introduced recently by Hidalgo, Hughes and Leyton--Álvarez in \cite{HHAL} and \cite{HHLA2026} (observe that this last paper does not require to work in characteristic zero), in relation to the following question. Let $M$ be a $d$-dimensional smooth projective variety with an action of a finite group $H:=(\IZ/k\IZ)^n$ for some non zero positive integers $d,n\geq 1$ and $k\geq 2$. Assume that the orbifold $M/H$ is the projective space $\IP^d$, and that the cover map $M\longrightarrow \IP^d$ with covering group $H$ has branch locus $n+1$ hyperplanes in $\IP^d$ in general position, $n\geq d+1$. If $H'$ is another group with the same property as $H$, then they show that $H'$ is equal to $H$ (up to some few exceptions) and they show that the manifold $M$ can be described in $\IP^n$ by a set of $(n-d)$ Fermat-type equations all of degree $k$ (see \cref{GFM}). In this situation we say that $M$ is {\it of type $(d;k,n)$}. Under certain conditions, these manifolds are complete intersection Calabi--Yau manifolds and by construction have many automorphisms. These are good candidates to investigate Enriques manifolds and Log Enriques varieties. The main result of the paper give a complete description of the quotients under certain assumptions and is formulated as follows:

\begin{theorem}[{\cref{main}}]\label{main0}
  Let $M_n^k$ be a Calabi--Yau GFM of type $(d;k,n)$. Assume that $\varphi\in H$ is an automorphism of $M_n^k$ with empty fixed locus. Then $\varphi$ has order $k$ and we have: 
    \begin{itemize}
        \item [$(i)$] If $d=2$ and $k=4$, then $n=3$ and $M_3^4/\langle\varphi\rangle$ is a Log Enriques surface of index two with canonical singularities, whose minimal resolution is an Enriques surface.
        \item [$(ii)$] If $d$ is even and $k=2$, then $M_n^2/\langle\varphi\rangle$ is an Enriques manifold of index two, that is a special member in the family described in \cref{Enriques_manifold}.
        \item [$(iii)$] If $d\geq 4$ is even and $k\neq 2$, then $M_n^k/\langle\varphi\rangle$ is a Log Enriques variety of index two with terminal singularities.
        \item [$(iv)$] If $d$ is odd and $k$ is a prime number, then $M_n^k/\langle\varphi\rangle$ is a Calabi--Yau manifold.
        \item [$(v)$] If $d$ is odd and $k$ is not a prime number, then  $M_n^k/\langle\varphi\rangle$ is a Calabi--Yau variety with terminal singularities.
    \end{itemize}
\end{theorem}

As another main result of the paper, in \cref{cor:not_deformation} we show that the Enriques manifolds obtained as the Hilbert scheme of $n$ points on an Enriques surface and the example described in \cite[Theorem 3.1]{OS} are two different families. In fact, their Betti numbers are different and equal to $11$ and \(1\), respectively. We also compute the dimensions of their deformation spaces in \cref{thm:dimension_deformation_space}. As a final remark and future research direction, observe that the above theorem admits a natural generalization in which the assumption that the automorphism has an empty fixed locus is removed, this is the subject of further studies. More broadly, the construction of Enriques manifolds may be approached by considering quotients of products of tori, IHS and Calabi--Yau manifolds. Also this direction will be the subject of future investigation.

We finish the introduction by recalling some recent results on Enriques manifolds and Log Enriques varieties. 
In \cite{BGGG} the authors consider $OG10$--manifolds and, by using properties of the LLV--algebra, they show that it is not possible to obtain examples of Enriques manifolds that are quotients of $OG10$--manifolds (and their deformations). 
The generalization of the notion of Enriques surface to the singular setting was introduced by Zhang in \cite{log-Enriques_surfaces},
who called them Log Enriques surfaces, see \cref{defi_zhang} for a discussion on the definition. Then the generalization to the singular setting in higher dimension was introduced independently in \cite{LogEnriques} and \cite{DROTX}. The definitions are slightly different, and in \cite{DTX} the authors study the deformations of what they call {\it primitive Enriques varieties}. In the papers \cite{PS} and \cite{Gachet} the Morrison--Kawamata cone conjecture was proved for Enriques manifolds with universal cover an IHS manifold. In the papers \cite{raman} and \cite{billi} the authors respectively prove the global Torelli theorem and investigate the Nielsen realization problem for Enriques manifolds of the same type as before. Finally, the very recent paper \cite{FRTV} studies the Brauer group of some known families of Enriques manifolds. 

\bigskip

{\bf Outline:}
The paper is organized as follows: in \cref{prelim} we recall some facts on Enriques manifolds, CY manifolds, Log Enriques varieties and CY varieties. We give  also several examples. \cref{GFM} is devoted to introduce GFM's and in particular Calabi--Yau GFM's. We recall in this section several results of \cite{HHAL}. In \cref{auto_empty} we study automorphisms with empty fixed locus and in \cref{sing} we investigate the singularities of the quotients. One main ingredient here is the use of the Reid--Tai criterion, see \cite{Kollar_Sing}. In these sections we complete the proof of our main \cref{main0}. Finally, in \cref{families_enriques} we prove that the families constructed in \cite[6.1.5]{LogEnriques} and \cite[Theorem 3.1]{OS} are distinct and determine the dimensions of their respective deformation spaces.

\bigskip

{\bf AI disclosure:}  All mathematical results and arguments are the work of the authors, no generative AI tools were used.

\bigskip

{\bf Acknowledgements:} First we warmly thank Rub{\'e}n Hidalgo for pointing out his paper during a visit of the last author at the University of Temuco in Fall 2023 and Michela Artebani, Samuel Boissi\`ere, Enrico Fatighenti, Lie Fu, Alice Garbagnati, Alvaro Liendo, Lucas Veyssi\`ere, Zhixin Xie  for very helpful discussions. This paper was finished when the  third author was visiting ICERM in Providence during the trimester {\it Computations on K3 surfaces and related varieties} in Fall 2026. She warmly thanks ICERM for the great working conditions.

\section{Preliminaries}\label{prelim}
To make the paper more self contained we recall in this section some well known basic facts and we follow closely \cite{LogEnriques}. All the varieties we consider will be normal projective complex varieties.

Let $X$ be a normal projective complex variety, we denote the sheaf of Kähler differentials by $\Omega^{1}_{X}$, by $X_{\text{reg}}$ the regular locus of $X$ with its open embedding $\iota:X_{\text{reg}}\hookrightarrow X$, and by $\Omega_{X_{\text{reg}}}^j:=\bigwedge^{j}\Omega_{X_{\text{reg}}}^1$ the sheaf of $j$-forms.
For $j\in\NN $ such that $0\leq j\leq \dim X$, we define
\begin{align*}
\Omega_{X}^{[j]}:=\iota_*\Omega_{X_{\text{reg}}}^j.
\end{align*}
The sections in $\Omega_{X}^{[j]}$ are called \textit{reflexive forms on $X$}, see \cite{Canonical_3_folds,Stable_reflexive_sheaves,Dif-forms_logcanonical}. 

We denote in particular the \textit{canonical sheaf} by:
\begin{align*}
    \omega_{X}:=\Omega_{X}^{[\text{dim}X]},
\end{align*}
it is a divisorial sheaf whose associated linear equivalence class of Weil divisors is the \textit{canonical divisor} $K_X$ of $X$. By \cite[Corollary $8$]{Minimal_Reid}, the divisorial sheaf associated to a canonical divisor is the canonical sheaf i.e.
\begin{align*}
    \omega_{X}=\mathcal{O}_X(K_X),
\end{align*}
and we denote $\omega_{X}^{[j]}=\mathcal{O}_{X}(j K_X)$, for all $ j\geq 0$. 

Recall that a normal variety $X$ has \textit{rational singularities} if for any resolution of singularities $f:Z\to X$, one has $R^{>0}f_*\mathcal{O}_Z=0$. By \cite[Corollary $1.8$]{Extending_holomorphic_forms}, if $X$ is a normal variety with rational singularities then the pullback over $X_{\text{reg}}$ of any reflexive form extends regularly on any resolution of singularities.

A \textit{finite covering map} is a finite and surjective morphism $f:\tX \to X$ between normal projective varieties of the same dimension, we say that the finite covering map $f$ is \textit{quasi-\'etale} if $f$ is \'etale in codimension one. This means that there exists a closed subset $Z\subset \tX$ of codimension at least two such that the restriction $f|_Z:\tX\setminus Z\to X$ is \'etale, see \cite{etale}.

\begin{definition}
Let $X$ be a normal variety. Then we say that $X$ is a \textit{Calabi--Yau variety} (CY) if $\omega_X \cong \mathcal{O}_X$,  $X$ has at most rational singularities, and 
\begin{align*}
    H^0\left(X,\Omega_{X}^{[j]}\right)=0,\quad \text{for all } 0<j<\dim X.
\end{align*}
The unique (up to scalar) generator $v_X$ of $H^0\left(X,\Omega_{X}^{[\dim X]}\right)$ will be called (holomorphic) {\it volume form}. 
\end{definition}

\begin{remark}
  \begin{itemize}
      \item [$(i)$] 
      In some references, e.g.\ \cite[Definition~1.4]{Hoering_Peternell}, 
      canonical singularities are required instead of rational singularities. In our setting these two conditions are equivalent. Indeed, for a variety $X$ with $\omega_X\cong\OO_X$ the canonical divisor is trivial and in particular Cartier, and for a normal variety with Cartier canonical divisor one has that $X$ has rational singularities if and only if $X$ has canonical singularities, see \cite[Corollary~$5.24$]{kollar-mori}. See also \cite[Section~$2$]{LogEnriques}. 
      \item [$(ii)$] One should be careful here since to be coherent with the results of our paper we do not use the same definition of CY as in \cite{LogEnriques}, which is more restrictive than ours. 
  \end{itemize}                
\end{remark}

We recall the following proposition that we will use several time in the sequel:
\begin{proposition}[{\cite[Proposition~$5.20$]{kollar-mori}}]\label{klt}
Let $g \colon X' \to X$ be a finite morphism between $n$-dimensional 
normal varieties. Then $X'$ is klt if and only if $X$ is klt.
\end{proposition}
\begin{remark}
Observe that in particular \cref{klt} holds when $X'$ is smooth.  
\end{remark}

\begin{definition}Let $X$ be a variety (not necessarily smooth) and $\varphi\in\Aut(X)$, then we say that {\it $\varphi$ has empty fixed locus} if $\Fix_{\langle\varphi\rangle}(X):=\left\{x\in X\mid\varphi^t(x)=x,\,\,\mbox{for each}\,\, t\in\mathbb{Z}\right\}$ is empty. \end{definition}\begin{remark}Observe that this means that the automorphism itself has no fixed points but powers of $\varphi$ might have fixed points, and this is what will happen in many situations in this paper. \end{remark}

\begin{lemma}\label{Lefschetz_number}
    Let $Y$ be a CY manifold and let $v_Y$ be its volume form. Assume that $\varphi\in \Aut(Y)$ is an automorphism of finite order $k$ with empty fixed locus. Then $\varphi^*v_Y=(-1)^{\dim Y+1}v_Y$.
\end{lemma}

\begin{proof}
     Since $\Fix_\varphi(Y)=\emptyset$, we obtain that the Lefschetz number $L(\varphi)$ is zero and so by the holomorphic Lefschetz fixed-point formula we have
\begin{align*}
    0=L(\varphi)=\sum_{j=0}^{\dim Y}(-1)^{j}Tr\,{\varphi^{*}}(\overline{H^0(Y,\Omega_Y^{j})}).
\end{align*}
Since $Y$ is a CY manifold, we have that $H^0(Y,\Omega_Y^{j})=0$ for all $0<j<\dim Y$. Then
\begin{align*}
    0&=Tr\,{\varphi^{*}}(\overline{H^0(Y,\Omega_Y^{0})})+(-1)^{\dim Y}\, Tr\,{\varphi^{*}}(\overline{\underbrace{H^0(Y,\Omega_Y^{\dim Y })}_{{=\, \CC \cdot \bar v_Y}}})\\
    0&=1+(-1)^{\dim Y}\,\bar \varepsilon, \quad \text{where } \varphi^*v_Y=\varepsilon\cdot v_Y.
    \end{align*}
    This finally gives
$$   
\varepsilon=\bar\varepsilon=(-1)^{\dim Y+1},
$$ 
so that $\varphi^*v_Y=(-1)^{\dim Y+1}v_Y$.
\end{proof}

\begin{definition}\label{log_enriques}
\begin{itemize}
    \item[$(i)$] A $K$-\textit{torsion variety} is a normal projective variety $X$ with klt singularities and torsion (but not trivial) canonical divisor $K_X$. The \textit{index} of $X$ is the smallest positive integer $I$ such that $IK_X\sim 0$ (observe that by definition $I>1$).
    \item[$(ii)$] A \textit{logarithmic Enriques variety} (\textit{Log Enriques} for short) is a $K$-torsion variety $X$ such that the irregularity $q(X):=h^0(\Omega^{[1]}_{X})=0$. 
    \item[$(iii)$]A Log Enriques variety is of \textit{CY type} if it admits a cyclic quasi-\'etale cover by a CY variety.
    \item [$(iv)$]A smooth Log Enriques variety is an \textit{Enriques manifold}.
\end{itemize}
\end{definition}
\begin{remark}\label{rational_klt}
\begin{itemize}
    \item[$(i)$] Observe that, by \cite[Theorem 5.22]{kollar-mori}, klt singularities are rational.
    \item[$(ii)$] Observe that when $X$ is an Enriques manifold of CY type then the cyclic quasi-\'etale cover is in fact \'etale, this follows from the Zariski--Nagata purity theorem \cite{Nagata, Zariski}.
 \end{itemize}
\end{remark}
\begin{remark}\label{defi_zhang}
        We point out that our definition of Log Enriques surface is slightly different from the definition given by Zhang in \cite{log-Enriques_surfaces} since there the author assumes that the Log~Enriques surface has no ADE singularities and he allows the canonical divisor to be trivial. 
    \end{remark}
\begin{proposition}\label{log_enriques_CY_type}
    Let $Y$ be a CY manifold and let $v_Y$ be its volume form. Assume that $\varphi\in \Aut(Y)$ is an automorphism of finite order $k$ with empty fixed locus.  Let $X:=Y/\langle\varphi\rangle$ and assume that $\pi:Y\to X$ is a quasi-étale cover, then the following hold:
    \begin{itemize}
    \item[$(i)$] $X$ has at worst klt singularities (which are rational). 
        \item [$(ii)$] If $\dim Y$ is an even number, then $\varphi$ has even order and  $X$ is a Log Enriques variety of CY type and index 2.
        \item [$(iii)$] If $\dim Y$ is an odd number, then $X$ is a CY variety.
    \end{itemize}
\end{proposition}

\begin{proof}
The proof of $(i)$ is well known and it follows e.g. from \cite[Theorem $5.21$, $(1)$]{kollar-mori} and from our \cref{rational_klt}. 

    If we assume that $\dim Y$ is even, then by \cref{Lefschetz_number} we have that $\varphi^*v_Y=-v_Y$, so that $k$ must be an even number and observe that
    \begin{align*}
        \varphi^*v_Y^{\otimes 2}=v_Y^{\otimes 2}.
    \end{align*}
    We obtain that 
    \begin{align*}
        H^0(X,\omega_X)=&H^0(X_{\text{reg}},\omega_{X_{\text{reg}}})
        =H^0(Y,\omega_{Y})^{\langle\varphi\rangle}=\{0\},
    \end{align*}
    and, in a similar way:
    \begin{align*}
        H^0\left(X,\omega_{X}^{[2]}\right)=H^0\left(X_{\text{reg}},\omega_{X_{\text{reg}}}^{\otimes 2}\right)\cong H^0\left(Y,\omega_{Y}^{\otimes 2}\right)^{\langle\varphi\rangle}=\CC\cdot v_Y^{\otimes 2}
    \end{align*}
   and since $v_Y^{\otimes 2}$ is never zero (otherwise $v_Y$ would also vanish at some point and this is not the case), we obtain that $2K_X\thicksim0$. 
   Since $Y$ is a CY variety, we have that
    \begin{align*}
        H^0\left(Y,\Omega_Y^j\right)=0, \quad \text{for all }0<j< \dim Y,
    \end{align*} 
    and, as before, we have an identification
    \begin{align*}
        H^0\left(X,\Omega_X^{[j]}\right)=H^0\left(Y,\Omega_Y^j\right)^{\langle\varphi\rangle}.
    \end{align*}
    Hence, $H^0\left(X,\Omega_X^{[j]}\right)=0$ for all $0<j\leq \dim X$ and in particular $q(X)=0$. By \cref{log_enriques}, we obtain that $X$ is a Log Enriques variety of CY type and index 2. This shows $(ii)$.\\

    On the other hand, if we assume that $\dim Y$ is an odd number and proceed as above, we obtain that $\varphi^*v_Y=v_Y$ so that $K_X\thicksim 0$ 
    and $H^0\left(X,\Omega_X^{[j]}\right)=0$ for all $0<j< \dim X$. We then conclude that $X$ is a CY variety, proving $(iii)$.
\end{proof}

\begin{corollary}\label{Enriques_manifold_index2}
Let $X$ be an Enriques manifold of Calabi--Yau type, that is $X=Y/\langle\varphi\rangle$
where $Y$ is a CY manifold, $\varphi\in\Aut(Y)$ is an automorphism of finite order
$k$ with empty fixed locus, and $\pi\colon Y\to X$ is a quasi-étale cover. Then $\pi$ is
in fact \'etale, the order of $\varphi$ is $2$, and $X$ has index $2$.
\end{corollary}

\begin{proof}
Since $X$ is smooth, $\pi$ is \'etale by \cref{rational_klt}~$(ii)$ and hence
$\Fix_{\varphi^s}(Y)=\emptyset$ for all $1\leq s\leq k-1$. Then, writing
$\varphi^*v_Y=\varepsilon\, v_Y$ this gives $\varepsilon^s=(-1)^{\dim Y+1}$ for all
$1\leq s\leq k-1$, by \cref{Lefschetz_number}. Note that $\varepsilon=-1$ when $s=1$, since otherwise
\cref{log_enriques_CY_type}~$(iii)$ would give that $X$ is a CY variety. Assume now that
$k\geq 3$. Then $s=2$ lies in the above range and yields $1=\varepsilon^2=-1$, which is a
contradiction. Therefore $k=2$ and $X$ has index $2$ by \cref{log_enriques_CY_type}~$(ii)$. See also \cite[Proposition $2.1$, $(2)$]{BNWS} for a proof using the holomorphic Euler characteristic.
\end{proof}

\begin{example}\label{examples_log-Enriques}
We discuss in detail some examples of Log Enriques varieties of CY type and index $2$.
    \begin{enumerate}
    \item[$1)$] Consider the smooth Fermat quartic surface in $\PP^3$
    \begin{align*}
        X_4:=\left\{ x_1^4+x_2^4+x_3^4+x_4^4=0\right\}.
    \end{align*}
    This is a K3 surface which admits the automorphism
    induced by 
    \begin{align*}
        \varphi: \CC^4\to \CC^4, \quad \varphi(x_1,x_2,x_3,x_4)=\left(x_1,ix_2,i^2x_3,i^3x_4\right)=\left(x_1,ix_2,-x_3,-ix_4\right),
    \end{align*}
    which, by abuse of notation, we call again $\varphi$.
    Observe that $\varphi$ has order $4$ and a direct computation shows that $\Fix_{\varphi}(X_4)$ is empty. 
    One can compute directly the action on the $2$-form. Indeed, let $f:=x_1^4+x_2^4+x_3^4+x_4^4$ then in the local coordinates we have
    \begin{align*}
        v_{X_4}=\frac{dx_3\wedge dx_4}{\partial f / \partial x_2}=\frac{dx_3\wedge dx_4}{4x_2^3},
    \end{align*}
    then 
    \begin{align*}
       \varphi^*v_{X_4}=\frac{i}{i^3}\cdot\frac{dx_3\wedge dx_4}{4x_2^3}=-\frac{dx_3\wedge dx_4}{4x_2^3}=-v_{X_4}.
    \end{align*}
    Now, note that $(\varphi^2)^*v_{X_4}=v_{X_4}$ and $\Fix_{\varphi^2}(X_4)$ is not empty: in fact
    \begin{align*}
        \varphi^2([x_1:x_2:x_3:x_4])=[x_1:-x_2:x_3:-x_4].
    \end{align*}
    It is well known that $Y_4:=X_4/\langle\varphi^2\rangle$ has eight $A_1$-singularities and its minimal resolution is a K3 surface (see \cite{Nikulin}). We denote the minimal resolutions of $Y_4$ and $Z_4:=X_4/\langle\varphi\rangle$ by
    \begin{align*}
        \widetilde{\pi}_1:\widetilde{Y}_4\to Y_4\quad \text{and}\quad \widetilde{\pi}_2:\widetilde{Z}_4\to Z_4,
    \end{align*}
    respectively, and observe that $\widetilde{Y}_4$ contains $8$ disjoint $(-2)$ exceptional curves.
    
    We have the following maps
    \begin{align*}
       \xymatrix@R=2em@C=6em{  
        X_4
       \ar[dr]^{2:1}
       \ar[dd]_{4:1}
        & &
        \ar[dl]_{\widetilde{\pi}_1}\ar[dd]^{\langle\varphi\rangle/\langle\varphi^2\rangle}\widetilde{Y}_4\\
       &\ar[dl]_{2:1}
       Y_4&
       \\
       Z_4& & \ar[ll]_{\widetilde{\pi}_2}\widetilde{Z}_4}
    \end{align*}
    where the involution $\langle \varphi \rangle/\langle \varphi ^2\rangle$ has no fixed points on $Y_4$ and it exchanges two by two the $A_1$ singularities. This involution lifts to a fixed-point-free involution on $\widetilde{Y}_4$ so that the minimal resolution $\widetilde{Z}_4$ of $Z_4$ is in fact an Enriques surface. Then $Z_4$ is a Log Enriques surface by \cref{log_enriques}. One can also conclude by using our \cref{log_enriques_CY_type}.

   \item[$2)$] Now, consider the smooth Fermat hypersurface of dimension $4$ and degree $6$ in $\PP^5$
   \begin{align*}
       X_6:=\left\{x_1^6+x_2^6+x_3^6+x_4^6+x_5^6+x_6^6=0\right\}.
   \end{align*}
   This is a Calabi--Yau manifold. Consider the automorphism
   \begin{align*}
       \varphi: \CC^6 \to \CC^6, \quad 
       \varphi(x_1,x_2,x_3,x_4,x_5,x_6)=\left(x_1,\xi_6\hspace{0.5mm}x_2,\xi_6^2\hspace{0.5mm}x_3,\xi_6^3\hspace{0.5mm}x_4,\xi_6^4\hspace{0.5mm}x_5,\xi_6^5\hspace{0.5mm}x_6\right),
   \end{align*}
   where $\xi_6$ is a primitive $6$-th root of unity. 
   Again, a direct computation as in the previous example shows that $\Fix_{\varphi}(X_6)$ is empty. We can compute now again directly the action on the $4$--form. Indeed, let $f:=x_1^6+x_2^6+x_3^6+x_4^6+x_5^6+x_6^6$ then in the local coordinates we have that
   \begin{align*}
       v_{X_6}=\frac{dx_3\wedge dx_4\wedge dx_5\wedge dx_6}{\partial f/\partial x_2}=\frac{dx_3\wedge dx_4\wedge dx_5\wedge dx_6}{6 x_2^5},
   \end{align*}
   then 
   \begin{align*}
       \varphi^*v_{X_6}=\frac{\xi^2_{6}\cdot\xi^3_{6}\cdot\xi^4_{6}\cdot \xi^5_{6}}{\xi_{6}^{5}}\cdot \frac{dx_3\wedge dx_4\wedge dx_5\wedge dx_6}{6 x_2^5}=\xi_{6}^3\cdot v_{X_6}=-v_{X_6},
   \end{align*}
   as we expected (see \cref{Lefschetz_number}). Now, note that a direct computation shows that
   \begin{align*}
       \Fix_{\varphi^2}(X_6)&=\left\{[x_1:0:0:x_4:0:0]\in\PP^{5} \mid x_1^6+x_4^6=0\right\}\\
       &\cup \left\{[0:x_2:0:0:x_5:0]\in\PP^{5}\mid   x_2^6+x_5^6=0\right\}\\ & \cup \left\{[0:0:x_3:0:0:x_6]\in\PP^{5}\mid   x_3^6+x_6^6=0\right\},
   \end{align*}
   which are $18$ isolated fixed points, and 
   \begin{align*}
       \Fix_{\varphi^3}(X_6)&=\left\{[x_1:0:x_3:0:x_5:0]\in\PP^{5}\mid  x_1^6+x_3^6+x_5^6=0\right\}\\ &\cup \left\{[0:x_2:0:x_4:0:x_6]\in\PP^{5}\mid   x_2^6+x_4^6+x_6^6=0\right\},
   \end{align*}
   are two disjoint plane sextics. Observe that these fixed loci do not intersect. 
 Thus, we have the following quasi-\'etale maps
   \begin{align*}
       \xymatrix@R=2em@C=6em{  
        X_6
       \ar@/^1.0pc/[r]^{\pi_1}
       \ar[r]_{3:1}
       \ar[dr]^{2:1}
       \ar@/_1.0pc/[dr]_{\widetilde{\pi}_1}
        & 
       \dfrac{X_6}{\langle\varphi^2\rangle}
       \ar[r]_{2:1}
       \ar@/^1.0pc/[r]^{\pi_2}
        &
       \dfrac{X_6}{\langle\varphi\rangle}\\
       & 
       \dfrac{X_6}{\langle\varphi^3\rangle}\ar[ur]^{3:1}\ar@/_1.0pc/[ur]_{\widetilde{\pi}_2}
       & }.
   \end{align*}
Since all the quotient maps are finite morphisms, \cref{klt} gives that 
all the quotient varieties have klt singularities, and hence rational singularities 
by \cref{rational_klt}. By \cref{log_enriques_CY_type}, we conclude that 
$X_6/\langle\varphi\rangle$ and $X_6/\langle\varphi^3\rangle$ are 
Log Enriques varieties of CY type and index $2$, and 
$X_6/\langle\varphi^2\rangle$ is a CY variety.
  \end{enumerate}
\end{example}

\begin{example}
We now discuss examples of CY manifolds and varieties.
\begin{itemize}
    \item [$1)$] Consider the smooth hypersurface of dimension $3$ and degree $5$ in $\PP^4$
   \begin{align*}
       X_5:=\left\{x_1^5+x_2^5+x_3^5+x_4^5+x_5^5=0\right\}.
   \end{align*}
   Also, consider the following automorphism
   \begin{align*}
      \varphi: \CC^5 \to \CC^5, \quad  \varphi(x_1,x_2,x_3,x_4,x_5)=\left(x_1,\xi_5\hspace{0.5mm}x_2,\xi_5^2\hspace{0.5mm}x_3,\xi_5^3\hspace{0.5mm}x_4,\xi_5^4\hspace{0.5mm}x_5\right),
   \end{align*}
   where $\xi_5$ is a primitive $5$-th root of unity and as before observe that $\Fix_{\varphi^s}(X_5)$ is empty, for all $1\leq s<5$.
By \cref{log_enriques_CY_type}, we obtain that $X_5/\langle\varphi\rangle$ is a Calabi--Yau manifold and since $X_5$ is simply connected we conclude that $X_5/\langle\varphi\rangle$ has fundamental group $\ZZ/5\ZZ$.

   \item [$2)$]Now, consider the smooth hypersurface of dimension $7$ and degree $9$ in $\PP^8$
   \begin{align*}
       X_9:=\left\{x_1^9+x_2^9+x_3^9+x_4^9+x_5^9+x_6^9+x_7^9+x_8^9+x_9^9=0\right\}.
   \end{align*}
   Also, consider the following diagonal automorphism
   \begin{align*}
       \varphi: \CC^9 \to \CC^9, \quad \varphi=\diag\left(1\ ,\xi_9\ ,\xi_9^2\ ,\xi_9^3\  ,\xi_9^4\ ,\xi_9^5\ ,\xi_9^6\ ,\xi_9^7\ ,\xi_9^8\ \right),
   \end{align*}
   where $\xi_9$ is a primitive $9$-th root of unity. Note that $\Fix_{\varphi}(X_9)$ is empty. Observe that $\Fix_{\varphi^3}(X_9)$ is not empty: in fact
   \begin{align*}
       \varphi^3=\diag\left(1\ ,\xi^3_9\ ,\xi_9^6\ ,1\ ,\xi_9^3\ ,\xi_9^6\ ,1\ ,\xi_9^3\ ,\xi_9^6\ \right),
   \end{align*}
   and we obtain that
   \begin{align*}
       \Fix_{\varphi^3}(X_9)&=\left\{[x_1:0:0:x_4:0:0:x_7:0:0]\in \PP^8 \mid  x_1^9+x_4^9+x_7^9=0\right\}\\
       &\cup \left\{[0:x_2:0:0:x_5:0:0:x_8:0]\in \PP^8\mid  x_2^9+x_5^9+x_8^9=0 \right\}\\
       &\cup \left\{[0:0:x_3:0:0:x_6:0:0:x_9]\in \PP^8 \mid  x_3^9+x_6^9+x_9^9=0 \right\}.
   \end{align*}
   These are three disjoint plane curves of degree 9. Observe that for $s=1,2,4,5,7,8$ the fixed locus of $\varphi^s$ is empty,
   and for $s=3,6$ the fixed loci are not empty and coincide.  
   Then we have the following maps
   \begin{align*}
       \xymatrix@R=2em@C=6em{  
        X_9
       \ar@/^1.0pc/[r]^{\pi_1}
       \ar[r]_{3:1}
        & 
       \dfrac{X_9}{\langle\varphi^3\rangle}
       \ar[r]_{3:1}
       \ar@/^1.0pc/[r]^{\pi_2}
        &
       \dfrac{X_9}{\langle\varphi\rangle}
       }.
   \end{align*}
   All quotient maps are finite, hence both $X_9/\langle\varphi\rangle$ and $X_9/\langle\varphi^3\rangle$ have klt singularities and in particular rational singularities (see \cref{klt} and \cref{rational_klt}). Then, by \cref{log_enriques_CY_type}, we conclude that both quotients are CY varieties.
\end{itemize}
    
\end{example}

\section{Generalized Fermat manifold of type Calabi--Yau}\label{GFM}
Throughout this section, we use the notation and conventions established in \cite{HHAL}. We recall that a {\it generalized Fermat pair} of type $(d;k,n)$ 
consists of a $d$-dimensional smooth complex algebraic variety $M$ together with a group $H\cong (\ZZ/k\ZZ)^n$ of automorphisms such that $M/H$ is the orbifold whose underlying space is the $d$-dimensional complex projective space $\PP^d$ branched along the complete intersection of $n+1$ hyperplanes in general position, each of branch order $k$. In this case we say that $M$ is a {\it generalized Fermat manifold (GFM) of type $(d;k,n)$}. Our main interest is in those GFM that are Calabi--Yau manifolds (we call them {\it Calabi--Yau GFM of type $(d;k,n)$}) and this happens, by using adjunction, when $k(n-d)-n-1=0$ (see also \cite[Proposition~$2$, $(3)$]{HHAL}). As a preliminary step, we recall some basic properties of hyperplanes in general position, which play a key role in 
the definition of generalized Fermat pairs.

Let $d\geq1$, $k\geq 2$, and $n\geq d+1$ be integers. First, observe that each hyperplane $L\subset \PP^d$ has the form
\begin{align*}
    L_q:=\{\rho_1t_1+\dots+\rho_{d+1}t_{d+1}=0\}, \quad \text{where } q:=[\rho_1,\dots, \rho_{d+1}]\in\PP^d.
\end{align*}
Consider now a set $q_1,\dots, q_{n+1}\in \PP^d$ of points and assume that $q_i\not= q_{j}$ for $i\not=j$ and $i,j\in \{1,\ldots,n+1\}$. For $d\geq 2$ and $3\leq s\leq d+1$, we say that the points $q_1,\dots, q_{n+1}\in \PP^d$ are in {\it general position} if any subset of $s$ elements of $\{q_1,\dots, q_{n+1}\}$ spans a $(s-1)$-plane $\Sigma\subset \PP^d$. 

A collection of $(n+1)$ hyperplanes $L_{q_1},\dots,L_{q_{n+1}}\subset \PP^d$ are in \textit{general position} if the corresponding points $q_1,\dots, q_{n+1}\in \PP^d$ are in general position (and so in particular pairwise different). Let $D$ be the divisor formed by the hyperplanes $L_{q_1},\dots,L_{q_{n+1}}$, then $D$ is a normal crossing divisor if the hyperplanes are in general position.

For $1\leq j\leq n+1$, let $q_j=[\rho_{1,j}:\dots:\rho_{d+1,j}]$. One can easily check that the hyperplanes $L_{q_1},\dots,L_{q_{n+1}}\subset \PP^d$ are in general position if and only if, for every $(d+1)\times(d+1)$-minors of the matrix
\begin{align*}
    M(q_1,\ldots,q_{n+1}):=\left(
    \begin{array}{ccc}
        \rho_{1,1} &\dots&\rho_{1,n+1}  \\
        \vdots& \ddots&\vdots\\
         \rho_{d+1,1} &\dots&\rho_{d+1,n+1}
    \end{array}
    \right)\in \text{M}_{(d+1)\times(n+1)}(\CC)
\end{align*}
are nonzero (in particular the matrix has rank $d+1$).

\begin{example}
    For example, let $e_1=[1:0:\dots:0]$, $e_2=[0:1:0:\dots:0]$, $\dots,$ $e_{d+1}=[0:\dots:0:1]$ and $e_{d+2}=[1:\dots:1]$ be points in $\PP^d$. Then the hyperplanes $L_{e_1},\dots,L_{e_{d+2}}\subset \PP^{d}$ are in general position.
\end{example}

We say that two ordered tuples $(L_{p_1},\dots, L_{p_{n+1}})$ and $(L_{q_1},\dots, L_{q_{n+1}})$ of hyperplanes in $\PP^d$ in general position are equivalents if there exists $T\in \PGL_{d+1}(\CC)$ such that $T(L_{p_i})=L_{q_{i}}$, for $i\in\{1,\ldots,n+1\}$. We denote by $X_{n,d}$ the set of such equivalence classes.

If $(L_{q_1},\dots, L_{q_{n+1}})$ is one of such tuples, then there is a unique $T\in\PGL_{d+1}(\CC)$ such that $T(L_{q_j})=L_{e_j}$, for $j\in\{1,\dots,d+2\}$. Then, for each $j\in\{d+3,\dots,n+1\}$ there is some unique $\Lambda_i:=[\lambda_{i,1}:\dots:\lambda_{i,d}:1]\in \PP^d$ for $i\in\{1,\dots,n-d-1\}$ such that $T(L_{q_j})=L_{\Lambda_{j-d-2}}$. In this case, we set $\Lambda:=(\lambda_1,\ldots,\lambda_d)\in\CC^{d(n-d-1)}$, where $\lambda_j=(\lambda_{1,j},\dots,\lambda_{n-d-1,j})\in\CC^{n-d-1}$, for $j\in\{1,\ldots,d\}$. So, if we set
\begin{align*}
    L_j(\Lambda)&:=L_{e_j}\subset \PP^d\quad \text{for }j\in\{1,\dots,d+2\}, \quad \text{and}\\
    L_j(\Lambda)&:=L_{\Lambda_{j-d-2}}\subset\PP^d \quad \text{for }j\in\{d+3,\dots,n+1\},
\end{align*}
then, the tuples $(L_1(\Lambda),\dots ,L_{n+1}(\Lambda))$ and $(L_{q_1},\dots,L_{q_{n+1}})$ are equivalent. We call the tuple $\Lambda$ a {\it standard parameter}.

Hence, we may identify 
\begin{itemize}
    \item [$(i)$] $X_{d+1,d}$ with the single point $(1,\ldots,1)\in\mathbb{C}^d$ and
    \item [$(ii)$] for $n\geq d+2$, $X_{n,d}$ with the set of tuples $\Lambda\in\CC^{d(n-d-1)}$ such that the hyperplanes $L_1(\Lambda),\dots ,L_{n+1}(\Lambda)$ are in general position. Note that $X_{n,d}$ is a non-empty open set.
\end{itemize}

Now, let $(M,H)$ be a generalized Fermat pair of type $(d;k,n)$ and $\pi:M\to \PP^d$ be an abelian branched covering with deck group $H$. Up to post-composition by some linear automorphism of $\PP^d$, we may assume that its branch locus is given by the hyperplanes $L_1(\Lambda),\dots,L_{n+1}(\Lambda)$ where $\Lambda:=(\lambda_1,\dots, \lambda_d)\in X_{n,d}$ and  $\lambda_i:=(\lambda_{1,i}\ ,\ldots,\lambda_{n-d-1,i}\ )\in\CC^{n-d-1}$ for $i=1,\ldots,d$. 

We will consider the complex projective algebraic set $M_n^k(\Lambda)$ in $\PP^n$ defined as follow
\begin{align}\label{equation_Fermat}
    M_n^k(\Lambda):=\left\{
    \begin{array}{rrrrrrrrc}
         x_1^k&+\dots+ &x_d^k &+&x_{d+1}^k&+&x_{d+2}^k&=&0  \\
        \lambda_{1,1} x_1^k&+\dots+ &\lambda_{1,d}x_d^k &+&x_{d+1}^k&+&x_{d+3}^k&=&0\\
         \vdots\hspace{0.25cm}&       & \vdots  \hspace{0.25cm}& &      \vdots\hspace{0.6cm}& &\vdots\hspace{0.60cm}  & &\vdots\\
         \lambda_{n-d-1,1}x_1^k&+\dots+ &\lambda_{n-d-1,d}x_d^k &+&x_{d+1}^k&+&x_{n+1}^k&=&0
    \end{array}
    \right.
\end{align}

The following propositions shows that $M_n^k(\Lambda)$ is a simply connected irreducible nonsingular complete intersection

\begin{proposition}[{\cite[Proposition~$1$, Remark~$3$]{HHAL}}]\label{connected_irred}
    If $\Lambda\in X_{n,d}$ then $M_n^k(\Lambda)$ defined in \eqref{equation_Fermat} is an irreducible nonsingular complete intersection. Moreover, if $d \geq 2$, 
then $M^k_n(\Lambda)$ is simply connected.
\end{proposition} 

Hence, we may and will assume that $M_n^k(\Lambda)$ has codimension $t := n-d$, in other words the nonsingular complete intersection $M_n^k(\Lambda)$ is defined by $t$ polynomials of Fermat type all with the same degree $k$. Now, let us consider the group $H_0=\langle\varphi_1,\ldots\varphi_{n+1}\rangle$, where 
$$
\varphi_i\left(\left[x_1:\dots: x_i:\dots: x_{n+1}\right]\right)=\left[x_1:\dots: \xi_kx_i:\dots: x_{n+1}\right]
$$ 
and $\xi_k$ is a primitive $k$-root of unity that is the same for all $i=1,\ldots,n+1$. Then an element $\varphi\in H_0$ is given as $\varphi:=\varphi_1^{m_1}\cdots \varphi_{n+1}^{m_{n+1}}$, where every $m_i\in \ZZ/k\ZZ$ and observe that $\varphi_1\circ\cdots\circ\varphi_{n+1}=\Id$. Moreover, we have that $H_0<\Aut(M_{n}^k(\Lambda))<\PGL_{n+1}(\CC)$.

As a consequence, any generalized Fermat pair $(M,H)$ of type $(d;k,n)$, 
where $M$ is a $d$-dimensional smooth variety and $H\cong(\ZZ/k\ZZ)^n$ is a 
group of automorphisms of $M$, is given as in \eqref{equation_Fermat} 
with the group $H_0$. This is stated in the following theorem:
\begin{theorem}[{\cite[Theorem~$1$]{HHAL}}]
    The pair $(M_n^k(\Lambda),H_0)$ is a generalized Fermat pair of type 
    $(d;k,n)$. Moreover, there exists a biholomorphism between $M$ and 
    $M_n^k(\Lambda)$ conjugating $H$ to $H_0$. In particular, every 
    generalized Fermat manifold of type $(d;k,n)$ whose associated 
    branching divisor is $\PGL_{d+1}(\CC)$-equivalent to the divisor 
    associated to $\Lambda$, is isomorphic to $M_n^k(\Lambda)$.
\end{theorem}

For simplicity from now on we denote by $M:=M_n^k(\Lambda)$, $H:=H_0$, we let $\Fix_{\varphi_i}(M)\subset M$ be the set of fixed points of $\varphi_i$, and set $$F(H):=\bigcup_{i=1}^{n+1}\Fix_{\varphi_i}(M).$$ 
Observe that $\Fix_{\varphi_i}(M)=\{x_i=0\}\cap M$ is a GFM of type $(d-1; k, n-1)$ for $d\geq2$, while for $d=1$ it consists of $k^{n-1}$ points by \cite[Section~$2.2.2$, $(IV)$]{HHAL}.
For $\varphi:=\varphi_1^{m_1}\cdots \varphi_{n+1}^{m_{n+1}}\in H$, we associate the tuple $m_{\varphi}:=(m_1,\ldots,m_{n+1})\in \left(\ZZ/k\ZZ\right)^{n+1}$. Then, for each $\ell\in \ZZ/k\ZZ$, we define
\begin{align*}
    L_\ell(m_{\varphi}):=\left\{j\in \{1,\ldots,n+1\}\mid m_j=\ell\right\},
\end{align*}
and the algebraic sets
\begin{align*}
    \widetilde{F}_\ell(m_{\varphi})=\left\{[x_1:\cdots:x_{n+1}]\in \PP^n \mid x_i=0,\  \text{for all }i\not\in L_\ell(m_{\varphi})\right\},
\end{align*}
and
\begin{align*}
    {F}_\ell(m_{\varphi})=\widetilde{F}_\ell(m_{\varphi})\ \cap\ M.
\end{align*}

The following proposition gives us the locus of fixed points on $M$ of every element of $H$:

\begin{proposition}[{\cite[Proposition~$3$]{HHAL}}]\label{fixed_points_1}
    Let $d\geq 1$, $n\geq d+1$, $k\geq 2$. Let $\varphi:=\varphi_1^{m_1}\cdots \varphi_{n+1}^{m_{n+1}}\in H$, different from the identity. Then its locus of fixed points is the disjoint union of the sets $F_\ell(m_{\varphi})$, where $\ell\in\ZZ/k\ZZ$. With the same notations as before, we have:
    \begin{itemize}
        \item [$(i)$] $F_\ell(m_{\varphi})\not=\emptyset$ if and only if $\# L_\ell(m_{\varphi})\geq n-d+1$.
        \item [$(ii)$] If $F_\ell(m_{\varphi})\not=\emptyset$, then it is a GFM of dimension $\# L_\ell(m_{\varphi})+d -n-1$.
    \end{itemize}
\end{proposition}

\begin{proposition}\label{fixed_points_2}
    Let $d\geq 1$, $n\geq d+1$, $k\geq 2$ and notations as before. Let $\varphi:=\varphi_1^{m_1}\cdots \varphi_{n+1}^{m_{n+1}}\in H$, different from the identity. In particular, the number of connected components of its locus of fixed points (if non empty) is equal to
        \begin{align*}
        \sum_{\ell\,\,\mbox{\tiny{such that}}\,\#L_\ell(m_\varphi)\geq n-d+1} c_\ell, \qquad 
        c_\ell=\begin{cases}1 & \text{if } \#L_\ell(m_\varphi)>n-d+1,\\ k^{n-d} & \text{if } \#L_\ell(m_\varphi)=n-d+1.\end{cases}
\end{align*}
  
\end{proposition}

\begin{proof}
   
   By \cref{fixed_points_1}, the locus of fixed points of $\varphi$ on $M$ is the disjoint union of the non-empty sets $F_\ell(m_\varphi)$, over  those $\ell \in \mathbb{Z}/k\mathbb{Z}$ with  $\#L_\ell(m_\varphi) \geq n-d+1$. We count the connected components of each $F_\ell(m_\varphi)$ separately.
   
   Assume first that $\#L_\ell(m_\varphi) > n-d+1$. Again by \cref{fixed_points_1}, $F_\ell(m_\varphi)$ is a GFM of positive dimension $\#L_\ell(m_\varphi) + d - n - 1 \geq 1$, hence irreducible by \cref{connected_irred}, and in particular connected. Therefore $c_\ell = 1$.

   Now, assume that $\#L_\ell(m_\varphi) = n-d+1$. Then $F_{\ell}(m_\varphi)$ has dimension $\#L_{\ell}(m_\varphi) + d - n - 1 = 0$, so it is a  finite set of isolated points, each of which is its own connected component. We show that $\#F_{\ell}(m_\varphi) = k^{n-d}$, so that $F_{\ell}(m_\varphi)$ contributes exactly $c_{\ell} = k^{n-d}$ connected components. We establish this claim by applying B\'ezout's theorem, as follows.

   Write $L_{\ell} (m_{\varphi}) = \{i_1,\ldots,i_{n+1-d}\}$ and introduce new coordinates $y_1 = x_{i_1}, \ldots, y_{n+1-d} = x_{i_{n+1-d}}$ for the variables whose indices belong to $L_{\ell}(m_{\varphi})$. Since $\#L_{\ell}(m_{\varphi})= n+1-d$, then  $$\widetilde{F}_{\ell}(m_{\varphi}) \cong \mathbb{P}^{\#L_{\ell}(m_{\varphi})- 1} = \mathbb{P}^{n-d},$$ with homogeneous coordinates $[y_1:\cdots:y_{n+1-d}]$.
   
   The variety $M$ is defined in $\mathbb{P}^n$ by exactly $n-d$ homogeneous equations of degree $k$ (see \eqref{equation_Fermat}):
   \begin{equation}\label{eq:system}
   \sum_{j=1}^{d} \lambda_{i,j}\, x_j^k + x_{d+1}^k + x_{d+2+i}^k = 0, \qquad i = 0, 1, \ldots, n-d-1,
   \end{equation} 
   where we put $\lambda_{0,j} = 1$ for all $j$. Upon restricting to $\widetilde{F}_{\ell}(m_{\varphi})$ (that is, setting $x_i = 0$ for all $i \notin L_{\ell}(m_{\varphi})$), each equation in \eqref{eq:system} becomes a homogeneous polynomial of degree $k$ in the variables $y_1, \ldots, y_{n+1-d}$. We thus obtain a system
   $$g_0(y_1,\ldots,y_{n+1-d}) = 0,\ \ldots,\ g_{n-d-1}(y_1,\ldots,y_{n+1-d}) = 0$$ 
   of $n-d$ equations of degree $k$ in $\mathbb{P}^{n-d}$, whose solution set is precisely $F_\ell(m_{\varphi}) = \widetilde{F}_{\ell}(m_{\varphi}) \cap M$.

   We verify that this system is a zero-dimensional complete intersection in $\mathbb{P}^{n-d}$, i.e., that the $n-d$ hypersurfaces $\{g_i = 0\}$ intersect in finitely many points. Indeed, since $\Lambda \in X_{n,d}$, the $n+1$ hyperplanes are in general position, which means that every $(d+1)\times(d+1)$ minor of the matrix of coefficients $M(q_1,\ldots,q_{n+1})$ is nonzero. This general position condition ensures that the restrictions $g_0,\ldots,g_{n-d-1}$ form a regular sequence in the polynomial ring $\mathbb{C}[y_1,\ldots,y_{n+1-d}]$, so their common zero locus in $\mathbb{P}^{n-d}$ has dimension $$\dim \mathbb{P}^{n-d} - (n-d) = (n-d) - (n-d) = 0\,,$$ as required.
    
    Since we have $n-d$ hypersurfaces of degree $k$ in $\mathbb{P}^{n-d}$ forming a zero-dimensional complete intersection, B\'ezout's theorem gives that their intersection consists of exactly $k^{n-d}$ points counted with multiplicity. Since $M$ is nonsingular and the intersection is transverse at every point (again by the general position condition on $\Lambda$), the fixed locus is also smooth and all multiplicities are equal to $1$, so $F_\ell(m_{\varphi})$ consists of exactly $k^{n-d}$ points. 
    
    Each of these $k^{n-d}$ points is an isolated point of $F_{\ell}(m_{\varphi})$, hence a connected component. Therefore $F_{\ell}(m_{\varphi})$ contributes with $c_{\ell} = k^{n-d}$ connected components as we wanted to show. Combining both cases gives the formula.
\end{proof}

We will consider here special GFM, which we call of \textit{Calabi--Yau} type:
\begin{proposition}\label{condition_CY}
    Let $n\geq d+1$, $d\geq 2$ and $k\geq 2$. Assume that $M$ is a GFM of type $(d;k,n)$. Then, $M$ is a Calabi--Yau manifold  if and only if $k(n-d)-n-1=0$. 
\end{proposition}
\begin{proof}
    The "if" direction is \cite[Proposition~$2$, $(3)$]{HHAL}. For the converse, it follows from a direct computation using the adjunction formula and the properties of the Hodge numbers of complete intersections in the projective space.
\end{proof}

\begin{remark}\label{k_d_odd}
    Observe that with $t=n-d$ and the Calabi--Yau condition we can write $n=tk-1$ and $d=t(k-1)-1$, so that:
    \begin{itemize}
        \item [$(i)$] $d$ will be an even number if and only if $k$ is even and $t\geq1$ is an odd number.
        \item [$(ii)$] $d$ will be an odd number if and only if $k$ is an odd number or $t\geq1$ is an even number.
    \end{itemize}
\end{remark}

\section{Automorphisms with empty fixed locus}\label{auto_empty}

With the same notations as before, let $M$ be a GFM of type $(d;k,n)$ with the equations given in \eqref{equation_Fermat}. Assume also that $M$ is of Calabi--Yau type and recall that we put $H=(\IZ/k\IZ)^n$, recall also that $t=n-d\geq 1$. 

\begin{proposition}\label{def_automorphism_empty}
Let $\varphi \in \Aut(M)$ be induced by the following automorphism of $\PP^n$, 
which by simplicity we call in the same way:
\begin{align*}
\varphi(x_i) = \xi_k^{r}\, x_i, \quad \text{for } i\in\{rt+1,\ldots,(r+1)t\},
\end{align*}
where $\xi_k$ is a primitive $k$-th root of unity and $r=0,\ldots, k-1$. 
Then $\varphi$ has empty fixed locus and any other automorphism in 
$H\subset \Aut(M)$ with empty fixed locus is of this form (up to permutation 
of the coordinates).
\end{proposition}

\begin{proof}
The definition of $\varphi$ shows $\#L_{\ell}(m_{\varphi})=t< n-d+1$, for each $\ell=0,\ldots, k-1$ and so by \cref{fixed_points_1}, we obtain that $\Fix_\varphi(M)=\emptyset$.

If $\psi$ is another automorphism with empty fixed locus such that there is an $l\in \{0,\dots,k-1\}$ with $\#L_l(m_\psi)<t$. Now, since $n+1=tk$ by the Calabi--Yau condition, there exists an $l'\in\{0,\ldots, k-1\}$ such that $\#L_{l'}(m_\psi)>t$. This contradicts the fact that $\psi$ has empty fixed locus (see \cref{fixed_points_1}).
\end{proof}
\begin{remark}\label{order}
Observe that the previous proposition tells us that if $M$ is GFM of Calabi--Yau type any automorphism in $H$ with empty fixed locus has order $k$.
\end{remark}
With the same notation as before, under the assumption of $\Fix_\varphi(M)=\emptyset$, we will now compute $\Fix_{\varphi^s}(M)$ for every $1<s<k$. We denote the order of $\varphi^s$ by $o(\varphi^s)$ and we have that 
\begin{align*}
    o\left(\varphi^s\right)=k/\gcd(s,k)\,, \quad \text{for every } 1<s<k\,.
\end{align*}

 We will show that in order to compute $\Fix_{\varphi^s}(M)$ for every $1<s<k$, it suffices to consider those values $s$ that are divisors of $k$. This claim is established in the following proposition and remark.

\begin{proposition}\label{fixed_points_automorphism}
 Let $M$ be a Calabi--Yau GFM of type $(d;k,n)$ and assume that $k$ is not a prime number. Let $\varphi\in H$ be an automorphism of $M$ with empty fixed locus and $S$ be the non-empty set:
 \begin{align*}
     S=\{s\in \ZZ\mid  1<s<k\,, \quad s \text{ divides }k \}\,.
 \end{align*}
 Then,
 \begin{align*}
     \Fix_{\varphi^s}(M)=\bigcup_{i=1}^{o\left(\varphi^s\right)}V_i^{(s)}, \quad \text{for every } s\in S
 \end{align*}
 and where the disjoint sets $V_i^{(s)}$ have dimension $t(s-1)-1$, for every $i\in\left\{1,\ldots,o(\varphi^s)\right\}$ and are also GFM of type $(t(s-1)-1;\,k,\,ts-1)$.
 \end{proposition}
\begin{proof} Recall that by \cref{order} the automorphism $\varphi$ has order $k$. Let $s\in S$, then $\gcd(k,s)=s$ and so $o(\varphi^s)=k/s$. Observe now that 
     \begin{align*}
         \xi_{k}^{\left({a+j \frac{k}{s}}\right)s}=\xi_{k}^{as},\quad \text{for }j\in\{0,..,s-1\}
    \end{align*}
     where $\xi_k$ is a primitive $k$-root of unity. So that we get for each $\ell=0,\ldots, k-1$. 
    \begin{align}\label{L}
        \#L_\ell(m_{\varphi^s})=st\geq t+1\,,
    \end{align}
    recall that here $s>1$ and $t\geq 1$. Observe that since $k$ is not prime, and $s\geq 2$, the equality holds if and only if $t=1$ and $s=2$. By \cref{fixed_points_1}, $\Fix_{\varphi^s}(M)$ is non-empty and the fixed locus in $\PP^n$ is the disjoint union 
    \begin{align*}
        \Fix_{\varphi^{s}}\left(\PP^n\right)=\bigcup_{i=1}^{o\left(\varphi^s\right)}\PP_{i}^{st-1}.
    \end{align*}
    We set $V_i^{(s)}:=\PP_{i}^{st-1}\cap M$. Then
    \begin{align*}
         \operatorname{dim}V_i^{(s)}=st-1-t=t(s-1)-1\geq 0\,, \quad \text{for every }i\in\{1,\ldots,o\left(\varphi^s\right)\}\, .
    \end{align*}
    Therefore, for every $s\in S$, the set of fixed points in $M$ is given by
    \begin{align*}
        \Fix_{\varphi^s}(M)=\bigcup_{i=1}^{o\left(\varphi^s\right)}V_i^{(s)}
    \end{align*}
    where $\dim V_i^{(s)}= t(s-1)-1$ and $V_i^{(s)}\cap V_j^{(s)}=\emptyset$, for $i\not=j$ since the $\PP_i^{st-1}$ are disjoint. Moreover, by \cref{fixed_points_1} and \cref{fixed_points_2}, every $V_i^{(s)}$ is a generalized GFM of type $(t(s-1)-1;\,k,\,ts-1)$.
\end{proof}
\begin{remark}\label{fixed_points_prime}
    By \cref{fixed_points_automorphism}, $k$ is a prime number if and only if $\Fix_{\varphi^s}(M)=\emptyset$ for every $1\leq s\leq k-1$. Indeed, if $k$ is prime then $\gcd(s,k)=1$ for every $1\leq s\leq k-1$, and hence $o(\varphi^s)=k=o(\varphi)$ yields that $\Fix_{\varphi^s}(M)=\Fix_\varphi(M)=\emptyset$. Conversely, if $k$ is not prime then the set $S$ is non empty, and \cref{fixed_points_automorphism} gives $\Fix_{\varphi^s}(M)\neq\emptyset$ for any $s\in S$.

Assume now that $k$ is not prime, and let $r$ be an integer with $1<r<k$ that does not belong to $S$. Then, standard computation gives the following cases
\begin{itemize}
    \item [$1.$] $\gcd(k,r)=1$, then $\Fix_{\varphi^{r}}(M)$ is empty.
    \item [$2.$] If it exists $s\in S$ such that $\gcd(k,r)=\gcd(k,s)=s$, then
    $\Fix_{\varphi^{r}}(M)=\Fix_{\varphi^{s}}(M)$, which was computed in \cref{fixed_points_automorphism}.
\end{itemize}
Therefore, to compute $\Fix_{\varphi^s}(M)$ for every $1<s<k$, it is enough to consider $s\in S$.
\end{remark}

\begin{example}
    Let us consider $k=6$ and $t=1$, then with the Calabi--Yau condition of \cref{{condition_CY}} we have $n=5$ and $d=4$. The set of proper divisors of $k$ is $S=\{2,3\}$, and 
    \begin{align*}
        M=\left\{
        x_1^6+x_2^6+x_3^6+x_{4}^6+x_{5}^6+x_{6}^6=0
        \right\}\subset \PP^5\,.
    \end{align*}
    We consider the automorphism $\varphi$ defined in \cref{def_automorphism_empty}:
    \begin{align*}
        \varphi=\diag(1,\xi_{6},\xi_{6}^{2},\xi_{6}^{3},\xi_{6}^{4},\xi_{6}^{5})\, ,
    \end{align*}
    where $\xi_6$ is a primitive $6$-th root of unity. By \cref{fixed_points_1}, note that $\Fix_\varphi(M)=\emptyset$ since $\#L_\ell(m_\varphi)=1<t+1=n-d+1=2$ for each $\ell=0,\ldots, 5$.
   
    A direct computation gives
    \begin{itemize}
        \item[$*$] $\varphi^2=\diag(1,\xi_{3},\xi_{3}^{2},1,\xi_{3},\xi_{3}^2)$, where $\xi_3=\xi_6^2$ is a primitive $3$-rd root of unity, with $o(\varphi^2)=3$ and $\#L_\ell(m_{\varphi^2})=2$\,  for $\ell=0,1,2\,;$
        \item[$*$] $\varphi^3=\diag(1,-1,1,-1,1,-1)$, with $o(\varphi^3)=2$ and $\#L_\ell(m_{\varphi^3})=3$\, for $\ell=0,1\,;$
        \item[$*$] $\varphi^4=\diag(1,\xi_{3}^{2},\xi_{3},1,\xi_{3}^{2},\xi_{3})$, with $o(\varphi^4)=3$ and $\#L_\ell(m_{\varphi^4})=2$\,  for $\ell=0,1,2\,;$
        \item[$*$] $\varphi^5=\diag(1,\xi_{6}^5,\xi_{6}^{4},\xi_{6}^{3},\xi_{6}^{2},\xi_{6})$, with $o(\varphi^5)=6$ and $\#L_\ell(m_{\varphi^5})=1$\, for $\ell=0,1,2,3,4,5\,.$
    \end{itemize}
    By \cref{fixed_points_prime}, in order to describe $\Fix_{\varphi^s}(M)$ for $1<s<6$ it is enough to consider $s\in S=\{2,3\}$. Indeed, $\gcd(6,4)=2$ gives $\Fix_{\varphi^4}(M)=\Fix_{\varphi^2}(M)$, and $\gcd(6,5)=1$ gives $\Fix_{\varphi^5}(M)=\Fix_\varphi(M)=\emptyset$. Moreover, since $\#L_{\ell}(m_{\varphi^s})=st\geq t+1$ for $s=2,3$ and any $\ell$, \cref{fixed_points_1} gives $\Fix_{\varphi^s}(M)\neq\emptyset$ in both cases. We now compute these two fixed loci using \cref{fixed_points_automorphism}.
    \begin{itemize}
        \item[$(i)$] $\Fix_{\varphi^2}(M)$: Note that $s=2$ and $o(\varphi^2)=3$. Since $\#L_{\ell}(m_{\varphi^2})=2=st$ for $\ell=0,1,2$, the fixed locus in $\PP^5$ is the disjoint union of $o(\varphi^2)=3$ projective lines, that we denote $\PP^1_j$, with $j=1,2,3$ so that 
              \begin{align*}
                  \Fix_{\varphi^2}(\PP^{5})=\PP_{1}^{1}\cup \PP_{2}^{1}\cup \PP_{3}^{1}=\bigcup_{i=1}^{o\left(\varphi^2\right)}\PP_{i}^{st-1}
              \end{align*}
    where $\PP_{1}^1=\left\{x_i=0\mid i\in I_1:=\{2,3,5,6\}\right\}$, $\PP_{2}^1=\left\{x_i=0\mid i\in I_2:=\{1,3,4,6\}\right\}$ and $\PP_{3}^1=\left\{x_i=0\mid i\in I_3:=\{1,2,4,5\}\right\}$. 
    Setting $V_j^{(s)}:=\PP_j^1\cap M$, each $V_j^{(s)}$ has dimension $t(s-1)-1=0$, so
        \begin{align*}
            \Fix_{\varphi^2}(M)=V_1^{(2)}\cup V_2^{(2)}\cup V_3^{(2)}=\bigcup_{i=1}^{o\left(\varphi^2\right)} V_i^{(2)}
        \end{align*}
         with the $V_j^{(2)}$ pairwise disjoint. Geometrically, each $V_j^{(2)}$ consists of the $6$ solutions of $x_a^6+x_b^6=0$ on the line $\PP_j^1$ with the indices $a,b\notin I_j$ and $a\neq b$, giving $18$ isolated fixed points in total, distributed $6$ on each of the three lines $\PP^1_j$.

        \item[$(ii)$] $\Fix_{\varphi^3}(M)$: Note that $s=3$ and $o(\varphi^s)=2$. Since $\#L_{\ell}(m_{\varphi^3})=3=st$ for $\ell=0,1$, the fixed locus in $\PP^5$ is the disjoint union of $o(\varphi^3)=2$ projective planes,
              \begin{align*}
                  \Fix_{\varphi^3}(\PP^{5})=\PP_{1}^{2}\cup \PP_{2}^{2}=\bigcup_{i=1}^{o\left(\varphi^3\right)}\PP_{i}^{st-1}
              \end{align*}
              where 
    $\PP_{1}^2=\left\{x_i=0\mid i\in\{2,4,6\}\right\}$ and $\PP_{2}^2=\left\{x_i=0\mid i\in\{1,3,5\}\right\}$. As above, setting $V_j^{(s)}:=\PP_j^2\cap M$, each $V_j^{(s)}$ has dimension $t(s-1)-1=1$, and
         \begin{align*}
             \Fix_{\varphi^3}(M)=V_1^{(3)}\cup V_2^{(3)}=\bigcup_{i=1}^{o\left(\varphi^3\right)} V_i^{(3)}\,,
         \end{align*}
         with $V_1^{(3)}\cap V_2^{(3)}=\emptyset$. Geometrically, $V_1^{(3)}$ and $V_2^{(3)}$ are two disjoint plane Fermat sextics.
    \end{itemize}
\end{example}

\begin{remark}\label{def_normal_variety}
   Assume that $k\geq 4$ is not a prime number and let $s$ be a divisor of $k$ (which means that $s\in S$). By \cref{fixed_points_automorphism}, we obtain that $v_s=\dim(\Fix_{\varphi^s}(M))=t(s-1)-1$, for every $s\in S$. Then
\begin{align*}
     k-s&\geq k/2\\
     t(k-1)-t(s-1)=t(k-s)&\geq k/2,\quad \text{since }t\geq1\\
     d-v_s=[t(k-1)-1]-[t(s-1)-1]&\geq k/2\geq 2, \quad \text{since }k\geq4.
\end{align*}
Thus $\dim\left(\Fix_{\varphi^s}(M)\right)\leq \dim M-2$, for every $s\in S$.
\end{remark}
\begin{proposition}\label{quasi_etale_morphism}
    Let $M$ be a Calabi--Yau GFM of type $(d; k,n)$. Assume that $\varphi\in H$ is an automorphism of $M$ with empty fixed locus, then $\pi:M\to M/\langle\varphi\rangle$
    \begin{itemize}
        \item [$(i)$] is an \'etale morphism if $k$ is a prime number.
        \item [$(ii)$] is a quasi-\'etale morphism if $k$ is not a prime number.
    \end{itemize}
\end{proposition}
\begin{proof}  Recall that by \cref{order} the automorphism $\varphi$ has order $k$. If we assume that $k$ is prime, then we have $\Fix_{\varphi^s}(M)$ is empty for every $s\in S$ (see \cref{fixed_points_prime}). Hence, $\pi:M\to M/\langle\varphi\rangle$ is an étale morphism proving $(i)$.

    On the other hand, if $k$ is not a prime number then we obtain that $\codim(\Fix_{\varphi^s}(M))\geq 2$ for every $s\in S$ (see \cref{def_normal_variety}), then $\pi:M\to M/\langle\varphi\rangle$ is a quasi-\'etale morphism.
\end{proof}

\begin{proposition}\label{quotients_log-enriques_CY}
    Let $M$ be a Calabi--Yau GFM of type $(d;k,n)$. Let $\varphi\in H$ be an automorphism of $M$ with empty fixed locus. Then, the following statements hold
    \begin{itemize}
        \item[$(i)$] If $d$ is an even number, then $M/\langle\varphi\rangle$ is a Log Enriques variety of CY type and index two, which is smooth if and only if $k=2$.
        \item[$(ii)$] If $d$ is an odd number, then $M/\langle\varphi\rangle$ is a CY variety, which is smooth if and only if $k$ is a prime number.
    \end{itemize}
\end{proposition}

\begin{proof}  Recall that by \cref{order} the automorphism $\varphi$ has order $k$. By \cref{quasi_etale_morphism}, the quotient map $\pi\colon M\to M/\langle\varphi\rangle$ is quasi-étale, so \cref{log_enriques_CY_type} applies. If $d$ is even, part $(ii)$ gives that $M/\langle\varphi\rangle$ is a Log Enriques variety of CY type and index two, and if $d$ is odd, part $(iii)$ gives that it is a CY variety.

For smoothness, by \cref{rational_klt}~$(ii)$ the quotient $M/\langle\varphi\rangle$ is smooth if and only if $\pi$ is \'etale, and by \cref{fixed_points_prime} this happens if and only if $k$ is prime. This equivalence holds for any $d$. In particular, if $d$ is odd, then $M/\langle\varphi\rangle$ is a CY manifold if and only if $k$ is prime, which proves $(ii)$. If $d$ is even, then $k$ is even by \cref{k_d_odd}~$(i)$ (see also \cref{log_enriques_CY_type}), so $k$ is prime if and only if $k=2$. This completes the proof of $(i)$, which is consistent with \cref{Enriques_manifold_index2} and \cref{Enriques_manifold} below.
\end{proof}

\begin{example}
  Let  $M$ be a Calabi--Yau GFM of type $(d;3,n)$. By \cref{k_d_odd}, we obtain that $n=3t-1$ and $d=2t-1$ for some $t\geq 1$, and in particular $d$ is an odd number. Let $\varphi$ be an automorphism of $M$ with empty fixed locus.  Recall that it is defined as:
    \begin{align*}
    \varphi([x_1:\dots:x_{3t}])=\left[x_1:\dots:x_{t}:\xi_{3}\hspace{0.25mm} x_{t+1}:\dots: \xi_{3}\hspace{0.25mm}x_{2t}:\xi_{3}^2\hspace{0.25mm} x_{2t+1}:\dots:\xi_{3}^2\hspace{0.25mm}x_{3t}\right]\, .
    \end{align*}
    By \cref{fixed_points_prime}, we have that $\Fix_{\varphi^2}(M)$ is empty too, so that by \cref{quotients_log-enriques_CY} we have  that $M/\langle\varphi\rangle$ is a CY manifold.
\end{example}
\begin{example}\label{Enriques_manifold}
   Let $M$ be a Calabi--Yau GFM of type $(d;2,n)$ of even dimension $d$. As before, we have that
    \begin{align*}
        n=2t-1\quad\text{and}\quad d=t-1, \ \ \text{for some odd}\,\,\, t\geq3, 
    \end{align*}
    and $M$ is given as follows
    \begin{align*}
         M=\left\{
    \begin{array}{rrrrrrrrc}
         x_1^2&+\dots+ &x_{t-1}^2 &+&x_{t}^2&+&x_{t+1}^2&=&0  \\
        \lambda_{1,1} x_1^2&+\dots+ &\lambda_{1,t-1}x_{t-1}^2 &+&x_{t}^2&+&x_{t+2}^2&=&0\\
         \vdots\hspace{0.2cm}&       & \vdots  \hspace{0.25cm}& &      \vdots\hspace{0.2cm}& &\vdots\hspace{0.2cm}  & &\vdots\\
         \lambda_{t-1,1}x_1^2&+\dots+ &\lambda_{t-1,t-1}x_{t-1}^2&+&x_{t}^2&+&x_{2t}^2&=&0
    \end{array}
    \right\} \,.
    \end{align*}
    The automorphism $\varphi$ of $M$ with empty fixed locus is
    \begin{align*}
        \varphi([x_1:\dots:x_{t}:x_{t+1}:\dots: x_{2t}])=[x_1:\dots:x_{t}:-\ x_{t+1}:\dots: -\ x_{2t}]
    \end{align*}
    which is unique up to permutation of the variables (see \cref{def_automorphism_empty}). 
    
   Then, it follows from \cref{quotients_log-enriques_CY} that $M/\langle\varphi\rangle$ is an Enriques manifold of CY type and index two with an \'etale morphism $\pi:M\to M/\langle\varphi\rangle$. 
\end{example}

\begin{remark}\label{subfamily_E}
    \cref{Enriques_manifold} is the unique example of an Enriques manifold that one can obtain from a Calabi--Yau GFM $M$ and an automorphism $\varphi\in\Aut(M)$ such that $\Fix_{\varphi^s}(M)$ is empty for every $s\in S$, (see also \cref{Enriques_manifold_index2} for the fact that the order of $\varphi$ must be exactly equal to $2$). This family of Enriques manifolds is a subfamily of the family obtained in \cite[$6.1.5$]{LogEnriques}.
\end{remark}

\section{Classification of Quotient Singularities}\label{sing}
 In this section, we will classify the quotient singularities of $M/\langle \varphi\rangle$, where $M$ is a Calabi--Yau GFM of type $(d;k,n)$ and $\varphi\in H$ is an automorphism of $M$ with empty fixed locus, recall that by \cref{order} this implies that the automorphism $\varphi$ has order $k$.
 
 Note that in \cref{fixed_points_prime}, we have shown that if $k\geq2$ is a prime number then $\Fix_{\varphi^s}(M)$ is empty, for every $0<s<k$. Here we will classify the type of the quotient singularities $M/\langle \varphi\rangle$, for every $k\geq 4$ that is a not a prime number.
 
Assume that $k\geq4$ is a not a prime number and recall the set of divisors of $k$ : 
\begin{align*}
    S=\{s\in\ZZ\mid  1<s<k \ \ \text{and}\ \ s\text{ divides }k \}\, .
\end{align*}

\begin{lemma}\label{fixed_points_in_others_divisors}
    Let $s,r\in \ZZ$ be divisors of $k$ and $1<s<r<k$. 
    \begin{itemize}
        \item [$(i)$]If $\gcd(s,r)\not=1$ then $\Fix_{\varphi^{\gcd(s,r)}}(M)\subseteq  \Fix_{\varphi^{s}}(M)\cap \Fix_{\varphi^{r}}(M)\, .$
        \item [$(ii)$]If $\gcd(s,r) = 1$, then $\mathrm{Fix}_{\varphi^s}(M) \cap \mathrm{Fix}_{\varphi^r}(M) = \emptyset\,.$
    \end{itemize}
\end{lemma}
\begin{proof}
     To prove $(i)$, set $\ell=\gcd(s,r)$. Then 
     \begin{align*}
         \ell\cdot s_1=s\quad \text{and}\quad \ell\cdot r_1=r,\quad \text{for some }s_1,r_1\ge1\, .
    \end{align*}
    Now, assume that $\ell\geq 2$ and that there exists $q\in \Fix_{\varphi^\ell}(M)$. Then, we have
    \begin{align*}
        \varphi^s(q)=\varphi^{\ell\cdot s_1}(q)=\underbrace{(\varphi^\ell\circ\cdots\circ\varphi^\ell)}_{s_1\text{-times}}(q)=q,
    \end{align*}
    and similarly we obtain that $\varphi^{r}(q)=q$. This proves $(i)$.
     
     Finally, assume that $\gcd(s,r)=1$. Then, there exist $a,b\in \ZZ$ such that $as + br = 1$. If $q \in \mathrm{Fix}_{\varphi^s}(M) \cap \mathrm{Fix}_{\varphi^r}(M)$, we have 
     \begin{align*}
         \varphi(q) = \varphi^{as + br}(q) = \varphi^{a\cdot s}(\varphi^{b\cdot r}(q)) = q\,,
     \end{align*}
     contradicting that $\mathrm{Fix}_{\varphi}(M) = \emptyset$, and this proves $(ii)$.
\end{proof}

\begin{definition}\label{def:stabilizer}
Let $q \in M$. Then the {\it stabilizer} of $q$ under the action of $\langle\varphi\rangle$ is defined as 
\begin{align*}
    \Stab(q) := \left\{\varphi^s\in \left\langle\varphi\right\rangle \mid \varphi^s(q) = q,\,\, s\in \IZ \,\right\}\,,
\end{align*}
which is a cyclic subgroup of $\left\langle\varphi\right\rangle$. For $q\in M$ with non-trivial stabilizer, we define the \textit{index of $q$} as
\begin{align*}
    s(q) := \min\left\{s \in S \mid q \in \Fix_{\varphi^s}(M)\right\}\,.
\end{align*}
\end{definition}

\begin{remark}\label{stab_relation}
    Observe that given $q\in M$, we have by \cref{def:stabilizer} that
    \begin{itemize}
        \item [$(i)$] $\Stab(q)$ is trivial if and only if $q\not\in \Fix_{\varphi^s}(M)$ for every $1\leq s\leq k-1$, in particular for every $s\in S$. 
        \item [$(ii)$] If $k$ is prime, then $\Stab(q)$ is trivial by the previous statement $(i)$ and \cref{fixed_points_prime}.
    \end{itemize}
\end{remark}

\begin{lemma}\label{stab}
Let $q \in \Fix_{\varphi^s}(M)$ where $s$ is the index of $q$, then
$\Stab(q) = \langle \varphi^s \rangle$. In particular the quotient
$M/\langle \varphi \rangle$ is locally isomorphic near $\pi(q)$ to $\AF^d / \langle \varphi^s
\rangle$ at the origin, where $\pi : M \to M/\langle \varphi \rangle$ is the
quotient map.
\end{lemma}

\begin{proof}
The inclusion $\langle \varphi^s \rangle \subseteq \Stab(q)$ follows from the definition of $\Fix_{\varphi^s}(M)$. For the other inclusion, let $\varphi^j(q)=q$ for any $1 \leq j \leq
k - 1$, and set $\ell := \gcd(j, k)$. Note that $\ell\geq 2$, otherwise if $\ell = 1$ then \cref{fixed_points_prime} gives $\Fix_{\varphi^j}(M) = \Fix_{\varphi}(M) = \emptyset$, contradicting $\varphi^j(q) = q$. 
Hence $\ell \in S$, and again by \cref{fixed_points_prime}, we obtain that $\Fix_{\varphi^j}(M) = \Fix_{\varphi^{\ell}}(M)\neq\emptyset$. In particular
$q \in \Fix_{\varphi^{\ell}}(M)$ with $\ell \in S$.

Now, we show that $s \mid \ell$. Since $q \in \Fix_{\varphi^s}(M) \cap
\Fix_{\varphi^{\ell}}(M)$ with $s, \ell \in S$, if $\gcd(s, \ell) = 1$ then
\cref{fixed_points_in_others_divisors},~$(ii)$ gives $\Fix_{\varphi^s}(M) \cap
\Fix_{\varphi^{\ell}}(M) = \emptyset$, contradicting $q$ lying in both. Hence
$\gcd(s, \ell) \neq 1$, and in particular $\gcd(s, \ell) \in S$ divides both
$s$ and $\ell$. By minimality of $s$, we have $s \leq \gcd(s, \ell)$ and since $\gcd(s, \ell)$ divides $s$,
we obtain $\gcd(s, \ell) = s$, that is, $s \mid \ell$.

Therefore $s \mid \ell = \gcd(j, k)$ and $\ell \mid j$, which means $\varphi^j \in \langle
\varphi^s \rangle$. This shows $\Stab(q) \subseteq \langle \varphi^s \rangle$,
completing the proof of the first claim.

Finally, we know that $M$ is smooth at $q$ and $\Stab(q) = \langle
\varphi^s \rangle$ acts linearly on the tangent space $T_q M \cong \AF^d$,
by \cite[Section~$3.17$ and Claim~$3.17.1$]{Kollar_Sing} we have the isomorphism
\begin{align*}
    \left(M/\langle \varphi \rangle,\, \pi(q)\right)
    \;\cong\;
    \left(\AF^d/\langle \varphi^s \rangle,\, 0\right)\,. 
\end{align*}
\end{proof}

\begin{lemma}\label{stab_subgroup}
Let $s\in S$ and let $q\in \Fix_{\varphi^{s(q)}}(M)$, where $s(q)$ is the index of $q$. Then
\begin{align*}
    \Stab_{\langle\varphi^s\rangle}(q)
        =\left\langle\varphi^{\lcm(s,s(q))}\right\rangle .
\end{align*}
Moreover $\lcm(s,s(q))$ divides $k$, so either $\lcm(s,s(q))\in S$ or $\lcm(s,s(q))=k$. If $s\in S$, then
$q\in\Fix_{\varphi^{\lcm(s,s(q))}}(M)$ and $M/\left\langle\varphi^s\right\rangle$ is locally isomorphic near $\pi(q)$ to
$\AF^d/\left\langle\varphi^{\lcm(s,s(q))}\right\rangle$ at the origin. If $\lcm(s,s(q))=k$, the stabilizer
$\Stab_{\langle\varphi^s\rangle}(q)$ is trivial and the image of $q$ is a smooth point
of $M/\langle\varphi^s\rangle$.
\end{lemma}

\begin{proof}
By \cref{stab} the stabilizer of $q$ in the full group is
$\Stab_{\langle\varphi\rangle}(q)=\langle\varphi^{s(q)}\rangle$. Since
$\langle\varphi^s\rangle\subseteq\langle\varphi\rangle$, we obtain
\begin{align*}
    \Stab_{\langle\varphi^s\rangle}(q)
        =\left\langle\varphi^s\right\rangle\cap\Stab_{\left\langle\varphi\right\rangle}(q)
        =\left\langle\varphi^s\right\rangle\cap\left\langle\varphi^{s(q)}\right\rangle
        =\left\langle\varphi^{\lcm(s,s(q))}\right\rangle
\end{align*}

As $s$ and $s(q)$ both divide $k$, so does $\lcm(s,s(q))$. Moreover $\lcm(s,s(q))\geq s>1$, so either
$\lcm(s,s(q))\in S$ or $\lcm(s,s(q))=k$. If $\lcm(s,s(q))=k$ then $\langle\varphi^{\lcm(s,s(q))}\rangle$ is trivial, the
action of $\langle\varphi^s\rangle$ is free near $q$, and its image in the quotient is a smooth point.
If $\lcm(s,s(q))\in S$, then $s(q)\mid\lcm(s,s(q))$ gives $\varphi^{\lcm(s,s(q))}(q)=q$, so
$q\in\Fix_{\varphi^{\lcm(s,s(q))}}(M)$. Since $M$ is smooth at $q$, the finite group
$\Stab_{\langle\varphi^s\rangle}(q)=\langle\varphi^{\lcm(s,s(q))}\rangle$ acts linearly on
$T_qM\cong\AF^d$, and by \cite[Section~3.17]{Kollar_Sing}, we obtain
\begin{align*}
    \bigl(M/\langle\varphi^s\rangle,\,\pi(q)\bigr)\;\cong\;\left(\AF^d/\langle\varphi^{\lcm(s,s(q))}\rangle,\,0\right)\,.
\end{align*}
\end{proof}

Now, remark that $o(\varphi^s)=k/s$ and by \cref{fixed_points_automorphism}, we obtain that  
\begin{align*}
    \Fix_{\varphi^s}(M)=\bigcup_{i=1}^{o(\varphi^s)}V_i^{(s)}, \quad \text{for all }s\in S
\end{align*}
 where the sets $V_i^{(s)}$ are disjoint of dimension $v_{s}=\dim(\Fix_{\varphi^s}(M))=t(s-1)-1$.

\begin{remark}\label{rmk:diagonalization}
Let $s\in S$ and $q\in\Fix_{\varphi^s}(M)$. Since $\varphi^s$ has finite order and acts on $T_qM\cong\AF^d$, this action can be diagonalized as
\begin{align*}
    \varphi^s=\diag(\underbrace{1,\ldots,1}_{v_s\text{ times}},
    \underbrace{\xi^{\sigma_1},\ldots,\xi^{\sigma_{d-v_s}}}_{d-v_s\text{ times}})
\end{align*}
where $\xi$ is a primitive $o(\varphi^s)$-th root of unity and
$1\leq\sigma_i\leq o(\varphi^s)-1$ for every $i$ (see \cite[Section~4.2]{age}).
\end{remark}

Moreover, we can classify the type of the quotient singularities using the following theorem. Recall that \(v_s\leq \dim(M)-2\) by \cref{def_normal_variety}.

\begin{theorem}[{\cite[Theorem~$4.11$]{age}}]\label{age}
    Let $\AF^m/\langle\varphi\rangle$ be a normal quotient singularity where we denote by $\AF^m$ the $m$-dimensional complex affine space. Let
    \begin{align*}
    \varphi=\diag({\xi^{\sigma_1},\ldots,\xi^{\sigma_{m}}})\,, 
\end{align*}
    where $\xi$ is a primitive $o(\varphi)$-th root of unity and $1\leq \sigma_i\leq o(\varphi)-1$, for every $i$. Then, $\AF^m/\langle\varphi\rangle$
    is terminal (respectively canonical) if and only if 
    \begin{align*}
       \alpha_j=\frac{1}{o(\varphi)}\sum_{i=1}^{m}\left[j\sigma_i\right]>1 \quad (\text{respectively} \ \ \geq 1)\, ,\quad \text{for all }1\leq j\leq o\left(\varphi\right)-1\,.
    \end{align*}
    where $\left[j\sigma_i\right]$ denotes the smallest residue$\pmod {o(\varphi)}$, and $\alpha_j$ is called the $\age$ of $\varphi^j$.
\end{theorem}

\begin{proposition}\label{1_4}
   Let $M$ be a Calabi--Yau GFM of type \((d;k,n)\), codimension $n-d=t=1$. Let $\varphi\in H$ be an automorphism of $M$ with empty fixed locus and order $4$. Then $M/\langle\varphi\rangle$ has canonical singularities.
\end{proposition}

\begin{proof}  Recall that by \cref{order} the automorphism $\varphi$ has order $k=4$. 
First observe that by using the Calabi--Yau assumption, we obtain that $n=3$ and $d=2$ (see \cref{k_d_odd}). Then, $M=X_4$ is the Fermat quartic, so that we could give a proof using \cref{examples_log-Enriques}~$(1)$, but we use \cref{age} to follow the same strategy as in the general case. Note that $s=2$ is the unique divisor of $k$ i.e $s\in S$ and by \cref{fixed_points_automorphism}, we obtain that $v_2=\dim(\Fix_{\varphi^2}(M))=t(s-1)-1=0$.

    Recall the $\varphi$ can be written as follows:
    \begin{align*}
        \varphi([x_1:x_2:x_3:x_4])=[x_1:i\hspace{0.25mm}x_2:-\hspace{0.25mm}x_3:-i\hspace{0.25mm}x_4]\ ,
    \end{align*}
    thus, we have that
    \begin{align*}
    \varphi^2([x_1:x_2:x_3:x_4])=\left[x_1:-\hspace{0.25mm}x_2:x_3:-\hspace{0.25mm}x_4\right].
    \end{align*}

By \cref{{Lefschetz_number}}, the action of $\varphi$ on the volume form is $\varphi^* v_{X_4}=-v_{X_4}$ so that $\varphi^2$ leaves it invariant. At $q\in\Fix_{\varphi^2}(X_4)$ in local coordinates we can write $\varphi^2$ in the diagonal form (see \cref{rmk:diagonalization})
\begin{align*}
        \left(\begin{array}{rr}
            -1 & 0 \\
             0 & -1
        \end{array}\ \right)\, .
    \end{align*}
One computes then the age of $\varphi^2$ as:
\begin{align*}
   \alpha_1= \frac{1}{2}(1+1)=1\, ,
\end{align*}
so by \cref{age} the quotient $\AF^2/\langle \varphi^2\rangle$ has canonical singularities (as expected since $\varphi^2$ is a symplectic involution on a K3 surface). Since $\varphi$ has empty fixed locus, we have $s(q) = 2$ for every 
$q \in \Fix_{\varphi^2}(M)$, and hence by \cref{stab} $\left(M/\left\langle\varphi\right\rangle, \pi(q)\right)$ is isomorphic to 
$(\mathbb{A}^2/\langle\varphi^2\rangle, 0)$. Therefore $M/\langle\varphi\rangle$ 
has canonical singularities.
\end{proof}

\begin{proposition}\label{terminal_singularities}
  Let $M$ be a Calabi--Yau GFM of type $(d;k,n)$. Assume that $k\geq 2$ and $(d,k)\not=(2,4)$. Assume that $\varphi\in H$ is an automorphism of $M$ with empty fixed locus. Then $M/\langle\varphi\rangle$ is smooth if and only if $k$ is a prime number, and otherwise it has terminal singularities. The same statements hold for every quotient $M/\langle\varphi^s\rangle$ with $1\leq s\leq k-1$.
\end{proposition}
\begin{proof}  Recall that by \cref{order} the automorphism $\varphi$ has order $k$. The assertion about the smoothness of the quotients is given in \cref{quotients_log-enriques_CY}. Now, assume that $k$ is not a prime number that means that $S$ is non empty. Let $s \in S$ and $q \in \Fix_{\varphi^s}(M)$. 
By \cref{rmk:diagonalization}, the action of $\varphi^s$ on $T_qM \cong \AF^d$ can be diagonalized as
$$
\varphi^s=\diag(1,\ldots,1,\xi^{\sigma_1},\ldots,\xi^{\sigma_{d-v_s}})
$$
where $\xi$ is a primitive $o(\varphi^s)$-th root of unity, $v_s=\dim\left(\Fix_{\varphi^s}(M)\right)$, and $1\leq\sigma_i\leq o(\varphi^s)-1 $ for every $i$. We can assume that the the smallest residue mod $o(\varphi^s)$ is $[\sigma_i]=\sigma_i$ for all $i=1,\ldots d-v_s$. 
This gives:
$$
\alpha^s_1:=\frac{1}{o(\varphi^s)}\sum_{i=1}^{d-v_s}\sigma_i\geq \frac{d-v_s}{o(\varphi^s)}=\frac{t(k-s)}{k}s\geq \frac{ts}{2}\,,
$$
since $d=t(k-1)-1$, $v_s=t(s-1)-1$ and $(k-s)/k\geq \frac{1}{2}$. We distinguish two cases:\\
{\bf Case 1}. If $ts>2$ then $\alpha^s_1>1$.\\
{\bf Case 2}. If $ts=2$, the only possibility is $t=1$ and $s=2$, since we are assuming that $S$ is non empty. Then, we obtain $\frac{2(k-2)}{k}\geq 1$. If we assume that
$\frac{2(k-2)}{k}=1$, then $k=4$ and we get $d=2$, which is not possible by assumption.
Therefore $\frac{2(k-2)}{k}>1$ and thus $\alpha^s_1>1$.\\
One can then repeat a similar computation for $j=2,\dots ,o(\varphi^s)-1$, since the action of $(\varphi^s)^j$ can be diagonalized locally in the same way, and we get that all the ages $\alpha^s_j>1$. 

Finally, by \cref{stab}, we have that $\Stab(q)=\langle\varphi^s\rangle$, for every $q \in \Fix_{\varphi^s}(M)$ whose index is $s\in S$. Hence, 
$(M/\langle\varphi\rangle, \pi(q))$ is isomorphic to $(\mathbb{A}^{d}/\langle\varphi^s\rangle, 0)$ and thus the singularity of $M/\langle\varphi\rangle$ at $\pi(q)$
coincides with that of $\AF^{d-v_s}/\langle\varphi^s\rangle$ at the origin, which is terminal by the
above computations and by \cref{age}. Since $\varphi$ has empty fixed locus, the singular points of
$M/\langle\varphi\rangle$ arise only from $\Fix_{\varphi^s}(M)$, where $s\in S$.
We conclude that $M/\langle\varphi\rangle$ has terminal singularities whenever
$(d,k)\neq(2,4)$.

The same conclusion holds for each quotient $M/\langle\varphi^s\rangle$ with $s\in S$. Indeed, by
\cref{stab_subgroup}, for a point $q\in \Fix_{\varphi^{s(q)}}(M)$ where $s(q)$ is its index, the stabilizer of $q$ in $\langle\varphi^s\rangle$ is
$\langle\varphi^{\gamma}\rangle$ with $\gamma=\lcm(s,s(q))$, and we have two cases:
\begin{itemize}
    \item [$(i)$] If $\gamma=k$, then $\langle\varphi^{\gamma}\rangle$ is trivial and the image of $q$
    is a smooth point of $M/\langle\varphi^s\rangle$.
    \item [$(ii)$] If $\gamma\in S$, then the singularity of $M/\langle\varphi^s\rangle$ at $\pi(q)$
    is isomorphic to that of $\AF^{d-v_{\gamma}}/\langle\varphi^{\gamma}\rangle$ at the origin. Its $\age$ can be computed in the same way as before with $\gamma$ in place of $s$, so $\alpha^{\gamma}_j>1$ for all $1\leq j\leq o(\varphi^{\gamma})-1$ whenever $(d,k)\neq(2,4)$. Hence $M/\langle\varphi^s\rangle$
    has terminal singularities whenever $(d,k)\neq(2,4)$.
\end{itemize}
Finally, for every $1\leq s\leq k-1$ we have $\langle\varphi^s\rangle=\langle\varphi^{\gcd(s,k)}\rangle$ with $\gcd(s,k)=1$ or $\gcd(s,k)\in S$, so $M/\langle\varphi^s\rangle$ equals either $M/\langle\varphi\rangle$ or $M/\langle\varphi^{\gamma}\rangle$ for some $\gamma\in S$. This concludes the proof. 
\end{proof}

We are now ready to show the main results of the paper. 

\begin{theorem}\label{main}
 Let $M_n^k$ be a Calabi--Yau GFM of type $(d;k,n)$. Assume that $\varphi\in H$ is an automorphism of $M$ with empty fixed locus. Then $\varphi$ has order $k$ and we have:
    \begin{itemize}
        \item [$(i)$] If $d=2$ and $k=4$, then $n=3$ and $M_3^4/\langle\varphi\rangle$ is a Log Enriques surface of index two with canonical singularities, whose minimal resolution is an Enriques surface.
        \item [$(ii)$] If $d$ is even and $k=2$, then $M_n^2/\langle\varphi\rangle$ is an Enriques manifold of index two, that is a special member in the family described in \cref{Enriques_manifold}.
        \item [$(iii)$] If $d\geq 4$ is even and $k\neq 2$, then $M_n^k/\langle\varphi\rangle$ is a Log Enriques variety of index two with terminal singularities.
        \item [$(iv)$] If $d$ is odd and $k$ is a prime number, then $M_n^k/\langle\varphi\rangle$ is a Calabi--Yau manifold.
        \item [$(v)$] If $d$ is odd and $k$ is not a prime number, then  $M_n^k/\langle\varphi\rangle$ is a Calabi--Yau variety with terminal singularities.
    \end{itemize}  
\end{theorem}

\begin{proof}
    The proof of this theorem follows from \cref{examples_log-Enriques}, \cref{quotients_log-enriques_CY}, \cref{1_4} and \cref{terminal_singularities}.
\end{proof}

With the same assumptions we can moreover show:
\begin{corollary} 
Assume that $k$ is not a prime number and $s\in S$.
\begin{itemize}
\item[$(i)$] If $d$ is odd then $M_n^k/\langle\varphi^s\rangle$ is a Calabi--Yau variety with terminal singularities.
\item[$(ii)$] If $d$ is even and $(d,k)\neq (2,4)$ then $M_n^k/\langle\varphi^s\rangle$ is a Calabi--Yau variety if and only if $s$ is even and it is a  Log Enriques variety of index two if and only if $s$ is odd. In both cases the singularities are terminal.
\item [$(iii)$] If $(d,k)=(2,4)$ then $n=3$ and $M_3^4/\langle\varphi^2\rangle$ is a singular K3 surface with $8$ singularties of type $A_1$.
\end{itemize}
\end{corollary}
\begin{proof}
This follows from \cref{terminal_singularities}, the proofs of \cref{Lefschetz_number} and \cref{log_enriques_CY_type}, and \cref{examples_log-Enriques}.

\end{proof}

\section{Two different families of Enriques manifolds of Calabi--Yau type}\label{families_enriques}
The aim of this section is to prove that the Enriques manifolds of CY type constructed in \cite[Theorem 3.1]{OS} and \cite[6.1.5]{LogEnriques} are not deformation equivalent. This will follow directly by the computation of their Betti numbers, for this we firstly recall some facts about cohomology of complete intersections and Hilbert schemes of points on surfaces.

\subsection{Cohomology of complete intersections}\label{Jabocian_ring}

We consider a smooth complete intersection \(X\subset \mathbb{P}^{n}\) of dimension \(\dim X={d}\). If \(\iota \colon X\hookrightarrow \mathbb{P}^{n}\) denotes the inclusion, then we have the pull-back (or restriction) map \(\iota
^*\colon H^*(\mathbb{P}^n,\mathbb{Q})\to H^*(X,\mathbb{Q})\) whose cokernel \(\operatorname{Coker}(\iota^*)=H_{prim}^*(X,\mathbb{Q})\) is called \textit{primitive cohomology} according to \cite[Definition 2.18]{Dimca_complete}. We recall the following standard fact about the cohomology of \(X\), which can be obtained with an iterated application of the Lefschetz hyperplane theorem, also proved for example in \cite{dimca1992singularities}.

\begin{lemma}\label{lem:primitive_cohomology_complete_intersections}
    Let \(X\subset \mathbb{P}^{n}\) be a smooth complete intersection with \(\dim X={d}\). The restriction morphism \(\iota^k\colon H^k(\mathbb{P}^{n},\mathbb{Q})\to H^k(X,\mathbb{Q})\) is an isomorphism for any \(k\not= d\), and it is injective for \(k=d\).
\end{lemma}

For any \(k\), the image of the morphism \(\iota^k\) consists of Hodge classes, hence \(H^k_{prim}(X,\mathbb{Q})\) carries a natural Hodge structure whose Hodge decomposition we denote by \(H^k_{prim}(X,\mathbb{C})=\bigoplus_{p+q=k}H^{p,q}_{prim}(X)\).
One observes in particular that there is no variation of Hodge structures in any degree \(k\not= d\). The middle cohomology \(H^d(X,\mathbb{Q})\) and its Hodge structure were studied by several authors, following Griffith's theory of residues.  We refer for example to \cite{terasoma1990infinitestimal,konno1991variational,Dimca_complete}. In order to be more self-contained, we recall here some facts about the Hodge structure of complete intersections, which are probably known to experts.

Let \(f_1,\dots,f_c\in \mathbb{C}[x_0,\dots,x_n]\) the defining equation for the smooth complete intersection \(X\subset \mathbb{P}^n\), set
\begin{equation*}
    S= \mathbb{C}[x_0,\dots,x_{n},z_1,\dots,z_{c}]
\end{equation*}
with bigrading given by \(\deg(x_i)=(0,1)\), \( \deg(z_i)=(1,-\deg(f_i))\), and let \(F=\sum_{i=1}^c f_i z_i\). In particular, if \(\dim X=d\) then \(d=n-c\). The \textit{Jacobian ideal} is the homogeneous ideal \begin{equation*}
     J=\left(\frac{\partial F}{\partial x_i},\frac{\partial F}{\partial z_k}\right)_{i=0,\dots n; j=1,\dots,c}
\end{equation*}and the \textit{Jacobian ring} is given by \begin{equation*}
    R=S/J.
\end{equation*}
Notice that \(\frac{\partial F}{\partial z_j}=f_j\) and \(F=\sum_{j=1}^c z_j \frac{\partial F}{\partial z_j}\). Moreover, we have \(\deg(\frac{\partial F}{\partial x_i})=(1,-1)\) and \(\deg(\frac{\partial F}{\partial z_j})=(0,\deg(f_j))\).
We denote by \(R_{a,b}\) and \(J_{a,b}\) the homogeneous components of bidegree \((a,b)\).
Remarkably, if \(\omega_X\cong \mathcal{O}_X(m)\), there is an isomorphism
\begin{equation*}
    R_{p,m}\cong H_{prim}^{d-p,p}(X),
\end{equation*}
which is given by 
\begin{equation*}
    [P]\mapsto \operatorname{Res}\left(\frac{P\Omega_x\wedge \Omega_z}{F^{c+p}}\right)
\end{equation*}
where we denote $x=(x_0,\ldots,x_n)$ and $z=(z_0,\ldots, z_n)$ and recall that for a set of variables \(u=(u_0,\dots, u_r)\) we denote 
\begin{equation*}
    \Omega_u=\sum_{i=0}^r (-1)^iu_i du_0\wedge \dots \wedge\widehat{du_i}\wedge\dots du_r.
\end{equation*} 
 This is explained for example in \cite[Section 2]{terasoma1990infinitestimal}, \cite[Pages 90-91]{Dimca_complete}, and \cite[2.2]{fatighenti2021note}.

We will be interested in the case where \(X\) is a CY, hence \(m=0\), and it is smooth complete intersection of \(c>1\) quadrics.
\begin{lemma}\label{lem:Euler_relation}
    With the above notation, assume that \(\deg(f_i)=2\) for \(i=1,\dots,c\). Then the only relation of linear dependence in \(J_{1,0}\) is given by 
    \begin{equation*}
        2\sum_{j=1}^cz_j\frac{\partial F}{\partial z_j}=\sum_{i=0}^n x_i\frac{\partial F}{\partial x_i}.
    \end{equation*}
\end{lemma}
\begin{proof} We may assume that \(X\not= \emptyset\), so that \(\dim X =d\geq 0\), or equivalently \(n-c>-1\), which gives \(n+1>c\).
    Since the dimension of \(R_{1,0}\) is invariant under deformation, also the one of \(J_{1,0}\) is, and we can choose \(F\) of a specific form. We let 
    \begin{equation*}
        f_j=\sum_{i=0}^n \lambda_{ij} x_i^2
    \end{equation*}
    and store the coefficients in a \((n+1)\times c\) matrix \(\Lambda=(\lambda_{ij})\). We may assume that any \(c\times c\) minor of \(\Lambda\) is not degenerate, and that any \(\lambda_{ij}\not=0\). In this case \((f_1,\dots,f_c)\) determine a smooth complete intersection, as explained in \cref{GFM}. The vector space \(J_{1,0}\) is generated by elements of the form \(x_i\frac{\partial F}{\partial x_k}\) or \(z_j\frac{\partial F}{\partial z_l}\).
    Assume that 
    \begin{equation*}
        \sum_{i,k=0}^n a_{ik}x_i \frac{\partial F}{\partial x_k}+\sum_{j,l=1}^c b_{jl}z_j \frac{\partial F}{\partial z_l}=0
    \end{equation*}
for \(a_{ik},b_{jl}\in \mathbb{C}\). This gives 
\begin{equation*}
    2\sum_{i,k=0}^n\sum_{j=1}^c a_{ik} \lambda_{kj} z_j x_i x_k+\sum_{j,l=1}^c\sum_{h=0}^n b_{jl} \lambda_{hl}z_jx_h^2=0
\end{equation*}
    which, factoring the \(z_j\)'s, implies 
    \begin{equation*}
         2\sum_{i,k=0}^n a_{ik} \lambda_{kj} x_i x_k= -\sum_{l=1}^c\sum_{h=0}^n b_{jl} \lambda_{hl}x_h^2
    \end{equation*}
    for any \(j=1,\dots,c\).
    This gives immediately \(a_{ik}\lambda_{kj}=0\) whenever \(i\not=k\) and, since \(\lambda_{kj}\not=0\), we get \(a_{ik}=0\). We form the matrices \(A=(a_{ik})\) and \(B=(b_{lj})\) of sizes \((n+1)\times(n+1)\) and \(c\times c\), respectively. We have just observed that \(A\) is diagonal, so we let \(a_i=a_{ii}\). We are left with the equality 
    \begin{equation*}
        2\sum_{i=0}^n a_{i} \lambda_{ij} x_i^2= -\sum_{l=1}^c\sum_{i=0}^n b_{jl} \lambda_{il}x_i^2
    \end{equation*}
that clearly gives 
\begin{equation*}
    2 a_{i} \lambda_{ij}= -\sum_{l=1}^c b_{jl} \lambda_{il}
\end{equation*}
for any \(i=0,\dots, n\), \(j=1,\dots,c\). This can be reinterpreted as the matrix equation \(2A\Lambda= -\Lambda B\). Since \(A\) is diagonal, this means that if \(r_i\) is the \(i\)-th row of \(\Lambda\), then \(r_i B= -2 a_i r_i\) and hence it is an eigenvector of \(B\). Denote by \(\mu_i\) the dimension of the eigenspace \(V_i\) of \(B\) relative to \(-2a_i\). Then \(\sum_{i=1}^k \mu_i=c\), where \(k\) is the number of distinct \(a_i\)'s, and in particular \(\mu_i\leq c\). Let \(m_i\) be the number of row vectors of \(\Lambda\) belonging to \(V_i\). We want to prove \(k=1\), for this assume by contradiction that \(k>1\). This implies \(\mu_i<c\) and that \(m_i\leq \mu_i\). In fact, if \(m_i\geq \mu_i+1\) then we can choose \(\mu_i+1\leq c\) independent row vectors in \(V_i\) of dimension \(\mu_i\), which is not possible. This gives \(n+1=\sum_{i=1}^k m_i\leq \sum_{i=1}^k \mu_i=c\), which is a contradiction since \(c<n+1\). In conclusion, \(k=1\) and \(B\) has only one eigenvalue. Since it has a basis of eigenvectors, we can assume that it is diagonal \(B=bI_{c}\). This gives \(r_i b=-2a_ir_i\) for any \(i\), hence \(A=aI_{n+1}\) and \(b=-2a\).
\end{proof}
We expect that this is the only relation in general.

\subsection{Cohomology of Hilbert schemes of points on surfaces}
We let \(S\) be a projective complex surface, and consider the Hilbert scheme \(S^{[n]}\) of \(n\geq 2\) points on \(S\). This is smooth by a result of Fogarty, of dimension \(2n\), and can be seen as the resolution of singularities of the symmetric product \(S^{(n)}\) via the Hilbert-Chow morphism \(S^{[n]}\to S^{(n)}\). We have the well-known following fact. 
\begin{lemma}\label{lem:Betti_2_Hilbert_scheme}
    Then the second Betti number satisfies \(b_2(S^{[n]})=\binom{b_1(S)}{2}+ b_2(S)+1\), for \(n\geq 2\).
\end{lemma}
This follows from the fact that \(H^2(S^{[n]},\mathbb{Q})\cong \bigwedge^2 H^1(S,\mathbb{Q})\oplus H^2(S,\mathbb{Q})\oplus \mathbb{Q}[E]\), where \(E\) is the exceptional divisor of the above resolution. 
The cohomology ring \(H^*(X,\mathbb{Q})\) has an explicit set of generators in terms of Nakajima operators \cite{nakajima1997heisenberg}, and formulas for Betti numbers and Hodge numbers were provided by G\"ottsche \cite{gottsche2006hilbert}.

\subsection{The Enriques manifolds of CY type}
Let us recall the construction of the two families, and fix some notations.
\begin{example}
    Following \cite[Theorem 3.1]{OS}, we consider an Enriques surface \(E\) and denote by \(X=E^{[n]}\) its Hilbert scheme of $n$ points for some \(n\geq 2\). The universal cover \(Y\to X\) has degree \(2\) and \(Y\) is a CY manifold, hence \(X\) is an Enriques manifold of CY type and index \(2\).
\end{example}
In particular, we observe that since $E$ is projective then \(Y\) is projective too. 
\begin{example}
    Following \cite[6.1.5]{LogEnriques}, we let \(Y'\subset \mathbb{P}^{4n+1}\) be the CY manifold of dimension \(\dim Y'=2n\) obtained as smooth complete intersection of \(2n+1\) general quadrics of the form \(Q(x)+Q'(y)\), where \(x\) denotes the first set of \(2n+1\) coordinates and \(y\) the remaining ones. The involution \(\varphi\in \Aut(Y')\) obtained as multiplication by \(-1\) on \(x\) is fixed-point-free, and it hence gives an Enriques manifold of CY type and index \(2\) as quotient \(Y'\to X'=Y'/\langle\varphi\rangle\). For a particular choice of quadrics, \(X'\) can be obtained as a quotient of a CY type generalized Fermat manifold as in \cref{Enriques_manifold}, see also \cref{subfamily_E}.
\end{example}

\begin{proposition}
    Let \(X=S^{[n]}\) be the Hilbert scheme of \(n\geq 2\) points on a smooth projective surface \(S\), and let \(X'\) be the free quotient of a smooth complete intersection of dimension \(2n\). Then \(X\) and \(X'\) are not deformation equivalent.  
\end{proposition}
\begin{proof}
    Observe that \(b_2(X)\geq b_2(S)+1>1\) by \cref{lem:Betti_2_Hilbert_scheme} and the fact that \(b_2(S)>0\) since \(S\) is projective. On the other hand, if \(X'=Y'/\langle\varphi\rangle \) for \(\varphi\in \Aut(Y')\), then \(H^*(X',\mathbb{Q})\cong H^*(Y',\mathbb{Q})^{\varphi^*}\) where the right hand side denotes the space of \(\varphi^*\)-invariant classes. This implies that \(b_2(X')\leq b_2(Y')=1\) by \cref{lem:primitive_cohomology_complete_intersections} (notice that equality is in fact attained whenever \(\varphi\) is restriction of a projective transformation). This immediately gives \(b_2(X)\not =b_2(X')\), implying that \(X\) and \(X'\) cannot be deformation equivalent by Ehresmann's Lemma. 
\end{proof}

From the previous general result we deduce that the given examples of Enriques manifolds of CY type are not deformation equivalent.

\begin{corollary}\label{cor:not_deformation}
    The Hilbert scheme \(X=E^{[n]}\) of \(n\geq 2\) points on an Enriques surface \(E\) and the free quotient \(X'\) of a smooth complete intersection of dimension \(2n\) of quadrics by an involution (as in \cite[6.1.5]{LogEnriques}) are not deformation equivalent.  
\end{corollary}
The Enriques manifolds of CY type \(X=E^{[n]}\), associated with Enriques surfaces $E$, are of particular interest for the following reason:
\begin{proposition}
The universal cover \(Y\to X=E^{[n]}\) of the Hilbert scheme of \(n\geq 2\) points on an Enriques surface \(E\), which is a projective CY manifold by \cite[Theorem 3.1]{OS}, is not deformation equivalent to a complete intersection.
\end{proposition}
\begin{proof}
    Since the morphism \(Y\to X\) is \'etale, of degree \(2\) by \cref{Enriques_manifold_index2}, there is an involution \(\varphi \in \Aut(Y)\) such that \(X=Y/\langle\varphi\rangle\). We have \(H^*(X,\mathbb{Q})\cong H^*(Y,\mathbb{Q})^{\varphi^*}\), which implies \(b_2(Y)\geq b_2(X)=11\) by \cref{lem:Betti_2_Hilbert_scheme} and the fact that \(b_2(E)=10\) and \(b_1(E)=0\) for an Enriques surface. For a complete intersection \(Y'\) we have \(b_2(Y')=1\) by \cref{lem:primitive_cohomology_complete_intersections}, and since \(b_2(Y')\not=b_2(Y)\) we deduce that \(Y'\) and \(Y\) cannot be deformation equivalent.
\end{proof}

\begin{remark}
The universal cover of \(E^{[2]}\) is explicitly described by Hayashi in \cite[Proposition 2.6]{hayashi2015universal} as a blow-up of a quotient of \(S^{2}\), where \(S\to E\) is the K3 universal cover of the Enriques surface. Its Hodge numbers are computed in \cite[Theorem 6.1]{hayashi2015universal}. For \(n>2\), only a birational model of the universal cover of \(E^{[n]}\) is provided by Hayashi.   
\end{remark}

In particular, we have that the CY coverings of the examples of Enriques manifolds given in \cite[Theorem 3.1]{OS} and \cite[6.1.5]{LogEnriques} are also not deformation equivalent. In fact, the universal cover of the first example is not deformation equivalent to a complete intersection by the previous proposition, but the universal cover of the second one is a complete intersection by construction.

\subsection{Dimension counts}
We now want to compute the dimensions of the deformation spaces of the previous examples.
Let's start by noticing that deformations of compact K\"ahler manifolds with torsion canonical bundle are unobstructed by \cite[Corollary 2]{Ran}. This applies to our cases, hence to understand the dimension of the deformation spaces of the examples of Enriques manifolds that we considered, we only need to compute the dimension of the first cohomology space associated to their tangent bundles.

As before, we let \(E\) be an Enriques surface, \(X=E^{[n]}\) and \(Y\to X\) its universal cover. 
We recall the following: the Kuranishi family of \(Y\) has base of dimension \(10\) by combination of \cite[Theorem 1.1]{hayashi2015universal} and \cite[Theorem 0.1 and Theorem 0.3]{fantechi1995deformation}.

We then turn our attention to the other family \(Y'\to X'\) where \(Y'\subset \mathbb{P}^{4n+1}\) is the smooth complete intersection of $(2n+1)$ quadrics admitting a free involution \(\varphi\) and \(X'=Y'/\langle\varphi\rangle\).

\begin{lemma}\label{lem:h1_tangent}
    We have that \(h^1(X',T_{X'})=h^1(Y',T_{Y'})-h^1(X',\Omega_{X'}^{2n-1})\).
\end{lemma}
\begin{proof}
    We follow the proof of \cite[Theorem 3.2]{hayashi2015universal}. Recall that since \(\omega_{Y'}\cong \mathcal{O}_{Y'}\) then \(T_{Y'}\cong \Omega_{Y'}^{2n-1}\). Moreover, \(\omega_{X'}\) is not trivial but \(\omega_{X'}^{\otimes 2}\cong \mathcal{O}_{X'}\) so that \(T_{X'}\cong \Omega_{X'}^{2n-1}\otimes \omega_{X'}\). Denote by \(\pi \colon Y'\to X'\) the covering map, 
    then 
    \begin{equation*}
        H^1(Y',\Omega_{Y'}^{2n-1})\cong H^1(X',\pi_*\Omega_{Y'}^{2n-1}).
    \end{equation*}
From \(Y'\cong \Spec(\mathcal{O}_{X'}\oplus \omega_{X'})\) we deduce
\begin{equation*}
    \pi_*\Omega_{Y'}^{2n-1}\cong \Omega_{X'}^{2n-1}\oplus( \Omega_{X'}^{2n-1}\otimes\omega_{X'}), 
\end{equation*}
which combined with the previous isomorphisms gives
\begin{equation*}
     H^1(Y',\Omega_{Y'}^{2n-1})\cong H^1(X',\Omega_{X'}^{2n-1})\oplus H^1(X',T_{X'})
\end{equation*}
from which we deduce the statement.
\end{proof}

\begin{proposition}
   Assume \(n\geq 2\). We have \(h^{1,2n-1}(X')=(2n-1)(2n+1)^2\) and \(h^{1,2n-1}(Y')=2h^{1,2n-1}(X')+1\). In particular, \(h^1(X',T_{X'})=(2n-1)(2n+1)^2+1\).
\end{proposition}
\begin{proof}
We make use of the Jacobian ring, hence we refer to results and notation of \cref{Jabocian_ring}, noticing that the ambient space in this case has dimension \(4n+1\) instead of \(n\). In this case we have that \(Y'\subset \mathbb{P}^{4n+1}\) is complete intersection of \(2n+1\) quadrics, hence \(d=4n+1\) and \(c=2n+1\).
Moreover, \(Y'\) is CY so that \(m=0\) and in this case we have an isomorphism 
\begin{equation*}
    R_{p,0}\cong H_{prim}^{2n-p,p}(Y').
\end{equation*}
We write a basis of \(R_{1,0}\) to compute the dimension of \(H^{1,2n-1}(Y')=H_{prim}^{1,2n-1}(Y')\), where the equality holds since \(2n-1\not=1\) for \(n\geq2\), and then determine the invariant part for the action of \(\varphi\) to get \(h^{1,2n-1}(X')\).

A basis of \(S_{1,0}\) is given by monomials of the form \(x_ix_jz_k\), and it has cardinality \((2n+1)\binom{4n+3}{2}\). Recall that \(\deg (\frac{\partial F}{\partial x_i})=(1,-1)\) and \(\deg(\frac{\partial F}{\partial z_k})=(0,2)\), so that a basis of the graded part \(J_{1,0}\) is given by elements of the form \(x_j\frac{\partial F}{\partial x_i}\), or \(z_j\frac{\partial F}{\partial z_i}\). There are \((4n+2)^2+(2n+1)^2\) such elements, but by \cref{lem:Euler_relation} there is precisely one linear dependence relation between them. Hence, we have 
\begin{equation*}
    h^{1,2n-1}(Y')=(2n+1)\binom{4n+3}{2}-(4n+2)^2-(2n+1)^2+1=4n(4n^2+2n-1)-1.
\end{equation*}
We now determine the monomials in \(S_{1,0}\) that are invariant under the action of \(\varphi\), observing that \(\varphi^*(\Omega_{x}\wedge\Omega_z)=-\Omega_{x}\wedge\Omega_z\) since \(\varphi\) acts as \(-1\) on \(2n+1\) variables, and \(\varphi^*(F)=F\). Hence, invariant cohomology classes correspond to anti-invariant monomials. We determine the invariant monomials modulo \(J\) and then deduce the number of anti-invariant monomials modulo \(J\).

Monomials that are not invariant are of the form \(x_ix_jz_k\) with \(i\in \{0,\dots,2n\}\) and \(j\in\{2n+1,\dots,4n+2\}\), hence the invariant monomials are \((2n+1)\binom{4n+3}{2}-(2n+1)(2n+1)^2=(2n+2)(2n+1)^2\), where \((2n+1)^2\) corresponds to the anti-invariant monomials \(x_ix_j\). The invariant elements in \(J_{1,0}\) are of the form \(x_i\frac{\partial F}{\partial x_j}\), or \(z_k\frac{\partial F}{\partial z_h}\) with \(i,j\in \{0,\dots,2n\}\), or \(i,j\in\{2n+1,\dots,4n+2\}\), and \(k,h\in \{1,\dots,2n+1\}\). We have a total of \(2(2n+1)^2+(2n+1)^2=3(2n+1)^2\) elements, with a relation of dependence as before (which is invariant). Hence we have 
\begin{equation*}
   (2n+2)(2n+1)^2-3(2n+1)^2+1=(2n-1)(2n+1)^2+1=2n(4n^2+2n-1)
\end{equation*}
invariant monomials in \(R_{1,0}\), and
\begin{equation*}
    h^{1,2n-1}(X')=4n(4n^2+2n-1)-1-2n(4n^2+2n-1)
\end{equation*}
which can be rewritten as 
\begin{equation*}
    h^{1,2n-1}(X')=2n(4n^2+2n-1)-1=(2n-1)(2n+1)^2.
\end{equation*}
From this the relation \(h^{1,2n-1}(Y')=2h^{1,2n-1}(X')+1\) is clear. By \cref{lem:h1_tangent}, \(h^1(X',T_{X'})=h^{1,2n-1}(Y')-h^{1,2n-1}(X')\) which then equals to \(h^{1,2n-1}(X')+1=2n(4n^2+2n-1)=(2n-1)(2n+1)^2+1\).
\end{proof}
\begin{remark}
For \(n=1\), \(Y'\) is a K3 surface and \(X'\) is an Enriques surface.
    The above computation for gives \(h^{1,1}(Y')-1=h^{1,1}_{prim}(Y')=19\) and \(h^{1,1}(X')-1=h^{1,1}_{prim}(X')=9\). 
\end{remark}

With the notation introduced above, we collect our results in the following theorem. 
\begin{theorem}\label{thm:dimension_deformation_space}
The deformation spaces of $E^{[n]}$ and of $X'$ have dimension respectively $10$ and $(2n-1)(2n+1)^2+1$ (recall $\dim X'=\dim E^{[n]}=2n$).
\end{theorem}

\begin{remark}
Observe that for $n=1$ we obtain the dimension of the deformation space of Enriques surfaces which is $10$ in both cases. 
Interestingly the dimension of the moduli space is independent on $n$ in the first case and dependent in the second case. 
\end{remark}

\begin{remark}
 We finish discussing the Morrison--Kawamata cone conjecture as formulated in \cite{PS} and in \cite{Gachet}. Observe that the Picard number of the Calabi--Yau manifold complete intersection and of its quotient, the Enriques manifold, is equal to one so that the Morrison--Kawamata cone conjecture trivially holds for both, the Calabi--Yau manifold and its Enriques quotient. In the case of the Enriques manifold $E^{[n]}$ and  of its universal Calabi--Yau cover the Picard number is bigger and the proof of the conjecture is open to our knowledge.
\end{remark}

\bibliographystyle{alpha}
\bibliography{bibliography}
\end{document}